\documentclass[10pt]{article}
\usepackage{etex}
\reserveinserts{28}
\usepackage[all]{xy}

\usepackage{amsfonts}
\usepackage{amsthm}
\usepackage{enumerate}
\usepackage{graphicx}
\usepackage{mathrsfs}
\usepackage{bm}
\usepackage{cite}
\usepackage{amssymb,amsmath} 
\usepackage{makecell}
\usepackage{tikz}
\usepackage{pgf}
\usepackage{tikz}
\usetikzlibrary{patterns}
\usepackage{pgffor}
\usepackage{pgfcalendar}
\usepackage{pgfpages}
\usepackage{shuffle,yfonts}
\usepackage{mathtools}
\DeclareFontFamily{U}{shuffle}{}
\DeclareFontShape{U}{shuffle}{m}{n}{ <-8>shuffle7 <8->shuffle10}{}

\newcommand{\ola}{\overleftarrow}
\newcommand{\ora}{\overrightarrow}

\newcommand{\yi}{{1}}

\newcommand\gD{{\Delta}}

\newcommand\eps{{\varepsilon}}
\newcommand{\bfp}{{\bf p}}
\newcommand{\bfq}{{\bf q}}
\newcommand{\bfr}{{\bf r}}
\newcommand{\bfu}{{\bf u}}
\newcommand{\bfv}{{\bf v}}
\newcommand{\bfw}{{\bf w}}
\newcommand{\bfy}{{\bf y}}
\newcommand{\bfe}{{\boldsymbol{\sl{e}}}}
\newcommand{\bfi}{{\boldsymbol{\sl{i}}}}
\newcommand{\bfj}{{\boldsymbol{\sl{j}}}}
\newcommand{\bfk}{{\boldsymbol{\sl{k}}}}
\newcommand{\bfl}{{\boldsymbol{\sl{l}}}}
\newcommand{\bfm}{{\boldsymbol{\sl{m}}}}

\newcommand{\bfx}{{\boldsymbol{\sl{x}}}}
\newcommand{\bfz}{{\boldsymbol{\sl{z}}}}
\newcommand\bfgl{{\boldsymbol \lambda}}

\newcommand\bfet{{\boldsymbol \eta}}

\newcommand\bfeps{{\boldsymbol \varepsilon}}
\newcommand\bfone{{\bf 1}}

 \allowdisplaybreaks
\usetikzlibrary{arrows,shapes,chains}
 \allowdisplaybreaks

\catcode`!=11
\let\!int\int \def\int{\displaystyle\!int}
\let\!lim\lim \def\lim{\displaystyle\!lim}
\let\!sum\sum \def\sum{\displaystyle\!sum}
\let\!sup\sup \def\sup{\displaystyle\!sup}
\let\!inf\inf \def\inf{\displaystyle\!inf}
\let\!cap\cap \def\cap{\displaystyle\!cap}
\let\!max\max \def\max{\displaystyle\!max}
\let\!min\min \def\min{\displaystyle\!min}
\let\!frac\frac \def\frac{\displaystyle\!frac}
\catcode`!=12

\let\oldsection\section
\renewcommand\section{\setcounter{equation}{0}\oldsection}

\allowdisplaybreaks

\DeclareMathOperator*{\dep}{dep}
\DeclareMathOperator{\Li}{Li}
\DeclareMathOperator{\Mi}{Mi}
\DeclareMathOperator{\ti}{ti}

\def\R{\mathbb{R}}

\def\N{\mathbb{N}}

\def\Q{\mathbb{Q}}
\def\CC{\mathbb{C}}

\def\su{\sum\limits_{n=1}^\infty}

\def\ze{\zeta}

\theoremstyle{plain}
\newtheorem{thm}{Theorem}[section]
\newtheorem{lem}[thm]{Lemma}
\newtheorem{cor}[thm]{Corollary}

\newtheorem{pro}[thm]{Proposition}
\theoremstyle{definition}
\newtheorem{defn}{Definition}[section]
\newtheorem{re}[thm]{Remark}

\begin{document}
\title{\bf Explicit Relations of Some Variants of Convoluted Multiple Zeta Values}
\author{
{Ce Xu${}^{a,}$\thanks{Email: cexu2020@ahnu.edu.cn, ORCID 0000-0002-0059-7420.}\ and Jianqiang Zhao${}^{b,}$\thanks{Email: zhaoj@ihes.fr, corresponding author, ORCID 0000-0003-1407-4230.}}\\[1mm]
\small a. School of Mathematics and Statistics, Anhui Normal University, Wuhu 241000, PRC\\
\small b. Department of Mathematics, The Bishop's School, La Jolla, CA 92037, USA}

\date{}
\maketitle

\noindent{\bf Abstract.} Kaneko and Yamamoto introduced a convoluted variant of multiple zeta values (MVZs) around 2016. In this paper, we will first establish some explicit formulas involving these values and their alternating version by using iterated integrals, which enable us to derive some explicit relations of the multiple polylogarithm (MPL) functions. Next, we define convoluted multiple $t$-values and multiple mixed values (MMVs) as level two analogs of convoluted MZVs, and, similar to convoluted MZVs, use iterated integrals to find some relations of these level two analogs. We will then consider the parametric MPLs and the parametric multiple harmonic (star) sums, and extend the Kaneko-Yamamoto's ``integral-series'' identity of MZVs to MPLs and MMVs. Finally, we will study multiple integrals of MPLs and MMVs by generalizing Yamamoto's graphical representations to multiple-labeled posets.



\section{Introduction}

We begin with some basic notations. Let $\N$ be the set of positive integers and $\N_0:=\N\cup \{0\}$.
A finite sequence $\bfk:=(k_1,\ldots, k_r)\in\N^r$ is called a \emph{composition}. We put
\begin{equation*}
 |\bfk|:=k_1+\cdots+k_r,\quad \dep(\bfk):=r,
\end{equation*}
and call them the weight and the depth of $\bfk$, respectively. If $k_1>1$, $\bfk$ is called \emph{admissible}.

For a composition $\bfk=(k_1,\ldots,k_r)$ and positive integer $n$, the \emph{multiple harmonic sums} (MHSs) and \emph{multiple harmonic star sums} (MHSSs) are defined by
\begin{align}
\zeta_n(\bfk):=\sum\limits_{n\geq n_1>\cdots>n_r>0 } \frac{1}{n_1^{k_1}\cdots n_r^{k_r}}\quad
\text{and}\quad
\zeta^\star_n(\bfk):=\sum\limits_{n\geq n_1\geq\cdots\geq n_r>0} \frac{1}{n_1^{k_1}\cdots n_r^{k_r}}\label{MHSs+MHSSs},
\end{align}
respectively. If $n<k$ then ${\zeta_n}(\bfk):=0$ and ${\zeta _n}(\emptyset )={\zeta^\star _n}(\emptyset ):=1$. When taking the limit $n\rightarrow \infty$ in \eqref{MHSs+MHSSs} we get the so-called the \emph{multiple zeta values} (MZVs) and the \emph{multiple zeta star values} (MZSVs), respectively
\begin{align*}
{\zeta}( \bfk):=\lim_{n\rightarrow \infty}{\zeta _n}(\bfk), \quad
\text{and}\quad
{\zeta^\star}( \bfk):=\lim_{n\rightarrow \infty}{\zeta^\star_n}( \bfk),
\end{align*}
defined for an admissible composition  $\bfk$ to ensure convergence of the series.

In general, let $\bfk=(k_1,\ldots,k_r)\in\N^r$ and $\bfz=(z_1,\dotsc,z_r)$ where $z_1,\dotsc,z_r$ are $N$th roots of unity. We can define the colored MZVs of level $N$ as
\begin{equation}\label{equ:defnMPL}
\Li_{\bfk}(\bfz):=\sum_{n_1>\cdots>n_r>0}
\frac{z_1^{n_1}\dots z_r^{n_r}}{n_1^{k_1} \dots n_r^{k_r}},
\end{equation}
which converges if $(k_1,z_1)\ne (1,1)$ (see \cite{YuanZh2014a} and \cite[Ch. 15]{Zhao2016}), in which case we call $({\bfk};\bfz)$ \emph{admissible}. The level two colored MZVs are often called Euler sums or alternating MZVs. In this case, namely,
when $(z_1,\dotsc,z_r)\in\{\pm 1\}^r$ and $(k_1,z_1)\ne (1,1)$, we set
$\ze(\bfk;\bfz)= \Li_\bfk(\bfz)$. Further, we put a bar on top of
$k_{j}$ if $z_{j}=-1$. For example,
\begin{equation*}
\ze(\bar2,3,\bar1,4)=\ze(2,3,1,4;-1,1,-1,1).
\end{equation*}
More generally, let $\bfk=(k_1,\ldots,k_r)\in\N^r$ and $\bfx=(x_1,\dotsc,x_r)$ where $x_1,\dotsc,x_r$ are complex variables.
The classical \emph{multiple polylogarithm} (MPL) and \emph{multiple polylogarithm star function}
with $r$ variables are defined by
\begin{align}
\Li_{\bfk}(\bfx):=\sum_{n_1>n_2>\cdots>n_r>0} \frac{x_1^{n_1}\dotsm x_r^{n_r}}{n_1^{k_1}\dotsm n_r^{k_r}}\quad
\text{and}\quad
\Li^\star_{\bfk}(\bfx):=\sum_{n_1\geq n_2\geq \cdots\geq n_r>0} \frac{x_1^{n_1}\dotsm x_r^{n_r}}{n_1^{k_1}\dotsm n_r^{k_r}},
\end{align}
respectively,
which converge if $|x_1\cdots x_{j}|<1$ for all $j=1,\dotsc,r$.
They can be analytically continued to a multi-valued meromorphic function on $\mathbb{C}^r$ (see \cite{Zhao2007d}). In particular, if $x_1=x,x_2=\cdots=x_r=1$, then $\Li_{k_1,\ldots,k_r}(x,\yi_{r-1})$ is the classical MPL with single-variable. As a convention, we denote by $\yi_d$ the sequence of 1's with $d$ repetitions.

The systematic study of MZVs began in the early 1990s with the works of Hoffman \cite{H1992} and Zagier \cite{DZ1994}. For an admissible index  $\bfk$, Hoffman called ${\zeta}( \bfk)$ multiple harmonic series while Zagier called ${\zeta}( \bfk)$ multiple zeta values since when $r=1$ they become Riemann zeta values $\zeta(k)$. Since then this area of research has attracted a lot of attentions due to their close connections with many other branches of mathematics and theoretical physics (see, for example, the book of the second author \cite{Zhao2016}).

We now give some additional definitions.

\begin{defn}
Let $\gD=\{z\in\CC:|z|\le 1\}$ be the unit disk.
Let ${\bfk}=(k_1,\ldots,k_r)\in\N^r$ and $\bfx=(x_1,\ldots,x_r)\in \gD^r$. We define the \emph{multiple harmonic sum} and \emph{multiple harmonic star sum} with $r$-variable by
\begin{align*}
\zeta_n(\bfk;\bfx):=\sum\limits_{n\geq n_1>\dotsm> n_r\geq 1}\frac{x_1^{n_1}\cdots x_r^{n_r}}{n_1^{k_1}\cdots n_r^{k_r}},\quad
\text{and}\quad
\zeta_n^\star(\bfk;\bfx):=\sum\limits_{n\geq n_1\geq\dotsm \geq n_r\geq 1}\frac{x_1^{n_1}\cdots x_r^{n_r}}{n_1^{k_1}\cdots n_r^{k_r}},
\end{align*}
respectively, where if $n<k$ then ${\zeta_n}(\bfk;\bfx):=0$ and $\ze_n(\emptyset;\emptyset)=\ze^\star_n(\emptyset;\emptyset):=1$.
\end{defn}

\begin{defn} (cf. \cite{KanekoYa2018})
For any two compositions of positive integers $\bfk=(k_1,\dotsc,k_r)$ and $\bfl=(l_1,\dotsc,l_s)$, define the Kaneko--Yamamoto
\emph{convoluted multiple zeta values} by
\begin{align}\label{equ:KYMZVs}
\zeta(\bfk\circledast\bfl^\star)&:=\sum\limits_{0<m_r<\cdots<m_1=n_1\geq \cdots \geq n_s>0}
\prod_{i=1}^r \frac{1}{m_i^{k_i}} \prod_{j=1}^s \frac{1}{n_{j}^{l_{j}}}
=\sum\limits_{n=1}^\infty \frac{\zeta_{n-1}(k_2,\ldots,k_r)\zeta^\star_n(l_2,\ldots,l_s)}{n^{k_1+l_1}}.
\end{align}
According the \cite[Thm. 4.1]{KanekoYa2018}, the convoluted MZVs can all be expressed as $\Q$-linear combinations of MZVs.
\end{defn}

The primary goals of this paper are to study the explicit relations of Kaneko-Yamamoto type convoluted MZVs and their related variants. It is not hard to see that the convoluted values of each type can be expressed as $\Q$-linear
combinations of values of the corresponding type by using the idea of posets as in Kaneko-Yamamoto's original work \cite{KanekoYa2018}
or by considering the more general form of Schur type as explained in our previous work \cite{XuZhao2020b}.

The remainder of this paper is organized as follows.

In section \ref{SEFIT}, we use the iterated integrals to establish several explicit formulas of MPLs
and multiple harmonic star sums with multi-variable.

In section \ref{SRCVMZV}, we apply the iterated integral identities obtained in Section \ref{SEFIT} to establish some explicit evaluations of convoluted MZVs, which in turn result in relations between MPLs and convoluted MZVs. Further, we also consider the convoluted multiple $t$-values and multiple mixed values by using iterated integrals.

We then define the parametric MPL (star) function and parametric multiple harmonic (star) sum in section \ref{SFPMPLS}, and extend some results of Sakugawa-Seki \cite[Thms. 2.10 and 2.13]{SS2016}. Finally, in section \ref{sec:mulInt}, we use the method of multiple-labeled posets to obtain the ``integral-series" identities of MPLs and multiple mixed values.

\section{Formulas of Some Iterated Integrals}\label{SEFIT}
The theory iterated integrals was developed first by K.T. Chen in the 1960's \cite{KTChen1971,KTChen1977}. It has played important roles in the study of algebraic topology and algebraic geometry in past half century. Its simplest form over $\R$ is
$$
\int_{a}^b f_p(t)dtf_{p-1}(t)dt\cdots f_1(t)dt:=\int\limits_{a<t_p<\cdots<t_1<b}f_p(t_p)f_{p-1}(t_{p-1})\cdots f_1(t_1)dt_1dt_2\cdots dt_p
$$
which can be easily extended to iterated path integrals over $\CC$.
Let $\bfk=(k_1,\ldots,k_r)\in\N^r$. From \cite[Eqs. (3.1) and (3.2)]{XuZhao2020b}, we have
\begin{align}\label{equ:glInteratedInt}
\Li_{\bfk}\left(x_1,\frac{x_2}{x_1},\dotsc,\frac{x_r}{x_{r-1}}\right)= \int_0^1 \left(\frac{x_r\, dt}{1-x_rt}\right)\left(\frac{dt}{t}\right)^{k_r-1}\cdots
\left(\frac{x_1\, dt}{1-x_1 t}\right)\left(\frac{dt}{t}\right)^{k_1-1}.
\end{align}
If changing $x_{j}$ to $xx_{j}$, then
\begin{align}\label{MPL-II-1}
&\Li_\bfk\left(xx_1,\frac{x_2}{x_1},\dotsc,\frac{x_r}{x_{r-1}}\right)= \int_0^x \left(\frac{x_r\,dt}{1-x_rt}\right)\left(\frac{dt}{t}\right)^{k_r-1}\cdots
\left(\frac{x_1dt}{1-x_1t}\right)\left(\frac{dt}{t}\right)^{k_1-1}.
\end{align}
Setting $x_1=\cdots=x_r=1$ in \eqref{MPL-II-1} we find that the single-variable MPL
\begin{align}\label{MPL-II-2}
&\Li_\bfk(x)= \int_0^x \left(\frac{dt}{1-t}\right)\left(\frac{dt}{t}\right)^{k_r-1}\cdots
\left(\frac{dt}{1-t}\right)\left(\frac{dt}{t}\right)^{k_1-1}.
\end{align}

To save space, for any composition $\bfk=(k_1,\dotsc,k_p)\in\N^p$ and $i,j\in\N$, we put
\begin{align*}
&\ora\bfk_{\hskip-2pt i,j}:=
\left\{
  \begin{array}{ll}
    (k_i,\ldots,k_{j}), \quad \ & \hbox{if $i\le j\le p$;} \\
    \emptyset, & \hbox{if $i>j$,}
  \end{array}
\right.
 \quad &\ola\bfk_{\hskip-2pt i,j}:=
\left\{
  \begin{array}{ll}
     (k_{j},\ldots,k_i), \quad\ & \hbox{if $i\le j\le p$;} \\
     \emptyset, & \hbox{if $i>j$.}
  \end{array}
\right.
\end{align*}
Set $\ora\bfk_{\hskip-2pt i}=\ora\bfk_{\hskip-2pt 1,i}=(k_1,\ldots,k_i)$ and $\ola\bfk_{\hskip-2pt i}=\ola\bfk_{\hskip-2pt i,p}=(k_p,\ldots,k_i)$ for all $1\le i\le p$.

\begin{thm}\label{thm-ItI1}
For any composition $\bfm=(m_1,\ldots,m_p)\in\N^p$, $n\in\N$, and $|x|\le 1$, we have
\begin{align}\label{FII1}
n\int_0^x t^{n-1}dt \frac{dt}{1-t} \left(\frac{dt}{t}\right)^{m_1-1}\cdots \frac{dt}{1-t}\left(\frac{dt}{t}\right)^{m_p-1}
=(-1)^p \ze^\star_n(\bfm;x)-\sum_{j=1}^{p}(-1)^{j} \ze^\star_n(\ora\bfm_{\hskip-2pt j-1})\Li_{\ola\bfm_{\hskip-1pt j}}(x).
\end{align}
where
\begin{equation*}
 \ze^\star_n(\bfm;x)= \ze^\star_n(\bfm;1_{p-1},x).
\end{equation*}
\end{thm}
\begin{proof}
We proceed by computing the multiple integral on the left-hand side of \eqref{FII1} as a repeated integral ``from left to right". Applying \eqref{MPL-II-2}, we have
\begin{align*}
&n\int_0^x t^{n-1}dt \frac{dt}{1-t} \left(\frac{dt}{t}\right)^{m_1-1}\cdots \frac{dt}{1-t}\left(\frac{dt}{t}\right)^{m_p-1}\\
&=\int_0^x \frac{t^ndt}{1-t} \left(\frac{dt}{t}\right)^{m_1-1}\cdots \frac{dt}{1-t}\left(\frac{dt}{t}\right)^{m_p-1}\\
&=\int_0^x \left(\frac{dt}{1-t}-\frac{(1-t^n)dt}{1-t}\right)\left(\frac{dt}{t}\right)^{m_1-1}\cdots \frac{dt}{1-t}\left(\frac{dt}{t}\right)^{m_p-1} \\
&=\Li_{\ola\bfm}(x)-\sum_{n_1=1}^n \int_0^x t^{n_1-1}dt \left(\frac{dt}{t}\right)^{m_1-1}\cdots \frac{dt}{1-t}\left(\frac{dt}{t}\right)^{m_p-1}\\
&=\Li_{\ola\bfm}(x)-\sum_{n_1=1}^n \frac1{n_1^{m_1}}\int_0^x \frac{t^{n_1}dt}{1-t} \left(\frac{dt}{t}\right)^{m_2-1}\cdots \frac{dt}{1-t}\left(\frac{dt}{t}\right)^{m_p-1}.
\end{align*}
We can now obtain the desired evaluation by repeatedly applying the above recurrence.
\end{proof}

\begin{thm}\label{thm-ItI2}
Suppose $\bfm=(m_1,\ldots,m_p)\in\N^p$, $n\in\N$, $|x|\le 1$ and $0<|\sigma_{j}|<1$ $(1\le j\le p)$.
Setting $\sigma_{p+1}=1/x$, we have
\begin{multline}\label{FII2}
n\int_0^x t^{n-1}dt \frac{dt}{1-\sigma_1t} \left(\frac{dt}{t}\right)^{m_1-1}\cdots \frac{dt}{1-\sigma_pt}\left(\frac{dt}{t}\right)^{m_p-1}\\
=\sum_{j=0}^{p}(-1)^{j} \frac{\ze^\star_n\left(\ora\bfm_{\hskip-2pt j};\frac{\sigma_1}{\sigma_2},\ldots,\frac{\sigma_{j}}{\sigma_{j+1}}\right)
    \Li_{\ola\bfm_{\hskip-2pt j+1}}\left(\sigma_px,\frac{\sigma_{p-1}}{\sigma_p},\ldots,\frac{\sigma_{j+1}}{\sigma_{j+2}}\right)}
{\sigma_1^{n+1}\sigma_2\cdots \sigma_p}.
\end{multline}
\end{thm}
\begin{proof}
The proof is completely similar to the proof of Theorem \ref{thm-ItI1} and is thus omitted.
\end{proof}

It is clear that if all $\sigma_{j}=1$ then Theorem \ref{thm-ItI2} becomes Theorem \ref{thm-ItI1}.

\begin{thm}
For any $(m_1,\ldots,m_p)\in \N\times \N_0^{p-1}$, $a_{j}\ne 1$ $(1\le j< p)$, and $|a_p|\le 1$,
\begin{align}\label{Int-G}
&n(-1)^p\int_{a_p}^1 \left(\frac{dt}{1-t}\right)^{m_p}\frac{dt}{a_{p-1}-t}\left(\frac{dt}{1-t}\right)^{m_{p-1}}\cdots \frac{dt}{a_1-t}\left(\frac{dt}{1-t}\right)^{m_1}t^{n-1}dt\nonumber\\
&=\sum_{j=1}^p (-1)^j \Li_{m_p+1,\ldots,m_{j+1}+1}\left(\frac{1-a_p}{1-a_{p-1}},\frac{1-a_{p-1}}{1-a_{p-2}},\ldots,\frac{1-a_{j+1}}{1-a_{j}} \right)\nonumber\\
&\quad\quad\quad\quad\quad\times \sum_{1\leq k_{|\widehat{\bfm}|_{j}}\leq \cdots \leq k_2\leq k_1\leq n} \frac{a_1^{k_{m_1}-k_{m_1+1}}\cdots a_{j-1}^{k_{|\widehat\bfm|_{j-1}}-k_{|\widehat\bfm|_{j-1}+1}}}{k_1k_2\ldots k_{|\widehat{m}|_{j}}} \left(1-a_{j}^{k_{|\widehat\bfm|_{j}}} \right).
\end{align}
where  and $|\widehat\bfm|_{j}:=m_1+\cdots+m_{j}+j-1$. In particular, $0^0$ should be interpreted as $1$ wherever it occurs. Here, when $j=1$ the numerator in the second sum is equal to $1$.
\end{thm}
\begin{proof}
The proof is done straightforwardly by computing the multiple integral on the left-hand side of \eqref{Int-G} as a repeated integral ``from right to left". First, setting $r=p-1$ in \eqref{MPL-II-1} we have
\begin{align*}
&\Li_{\bfk}\left(xx_1,\frac{x_2}{x_1},\dotsc,\frac{x_{p-1}}{x_{p-2}}\right)= \int_0^x \left(\frac{\, dt}{x_{p-1}^{-1}-t}\right)\left(\frac{dt}{t}\right)^{k_{p-1}-1}\cdots
\left(\frac{\, dt}{x_1^{-1}-t}\right)\left(\frac{dt}{t}\right)^{k_1-1}\\
&=\int_{1-x}^1 \left(\frac{dt}{1-t}\right)^{k_1-1}\frac{dt}{x_1^{-1}-1+t}\cdots\left(\frac{dt}{1-t}\right)^{k_{p-1}-1}\frac{dt}{x_{p-1}^{-1}-1+t}.
\end{align*}
Further, setting $x=1-a_p$, $x_{j}=(1-a_{p-j})^{-1}$ and $k_{j}=m_{p+1-j}$ $(1\le j<p)$ we have
\begin{align}\label{Int-G-Prof}
&-\Li_{\ola\bfm}\left( \frac{1-a_p}{1-a_{p-1}},\frac{1-a_{p-1}}{1-a_{p-2}},\ldots, \frac{1-a_2}{1-a_1} \right)\nonumber\\
&=(-1)^p\int_{a_p}^1 \left(\frac{dt}{1-t}\right)^{m_p-1}\frac{dt}{a_{p-1}-t}\cdots \left(\frac{dt}{1-t}\right)^{m_2-1}\frac{dt}{a_1-t}.
\end{align}
Noting that
\[\frac{a^m-t^m}{a-t}=\sum_{k=1}^m a^{m-k}t^{k-1}\]
we get
\begin{align*}
&n(-1)^p\int_{a_p}^1 \left(\frac{dt}{1-t}\right)^{m_p}\frac{dt}{a_{p-1}-t}\cdots \left(\frac{dt}{1-t}\right)^{m_2}\frac{dt}{a_1-t}\left(\frac{dt}{1-t}\right)^{m_1}t^{n-1}dt\nonumber\\
&=(-1)^p\int_{a_p}^1 \left(\frac{dt}{1-t}\right)^{m_p}\frac{dt}{a_{p-1}-t}\cdots \left(\frac{dt}{1-t}\right)^{m_2}\frac{dt}{a_1-t}\left(\frac{dt}{1-t}\right)^{m_1-1}\frac{1-t^n}{1-t}dt\nonumber\\
&=(-1)^p\sum_{1\leq k_{m_1}\leq \cdots \leq k_1\leq n} \frac{1}{k_1\cdots k_{m_1}}\int_{a_p}^1\left(\frac{dt}{1-t}\right)^{m_p}\frac{dt}{a_{p-1}-t}\cdots \frac{dt}{a_2-t}\left(\frac{dt}{1-t}\right)^{m_2}\frac{1-t^{k_{m_1}}}{a_1-t}dt.
\end{align*}
Writing $1-t^{k_{m_1}}=1-a_1^{k_{m_1}}+a_1^{k_{m_1}}-t^{k_{m_1}}$ and applying \eqref{Int-G-Prof}
we see that the above is equal to
\begin{align*}
&-\Li_{m_p+1,\ldots,m_2+1} \left( \frac{1-a_p}{1-a_{p-1}},\frac{1-a_{p-1}}{1-a_{p-2}},\ldots, \frac{1-a_2}{1-a_1} \right)\sum_{1\leq k_{m_1}\leq \cdots \leq k_1\leq n}\frac{1-a_1^{k_{m_1}}}{k_1\cdots k_{m_1}}\\
&\quad+(-1)^p \sum_{1\leq k_{m_1+1}\leq k_{m_1}\leq \cdots \leq k_1\leq n} \frac{a_1^{k_{m_1}-k_{m_1+1}}}{k_1\cdots k_{m_1}k_{m_1+1}} \\
&\quad\quad\quad\quad\quad\quad\quad\times\int_{a_p}^1\left(\frac{dt}{1-t}\right)^{m_p}\frac{dt}{a_{p-1}-t}\cdots \left(\frac{dt}{1-t}\right)^{m_3}\frac{dt}{a_2-t}\left(\frac{dt}{1-t}\right)^{m_2-1} \frac{1-t^{k_{m_1+1}}}{1-t} dt.
\end{align*}
We can now arrive at the desired formula by repeatedly applying the above recurrence.
\end{proof}

The above theorem has an interesting application to some type of duality result. For this, we recall that the Hoffman dual of a composition $\bfm=(m_1,\ldots,m_p)$ is $\bfm^\vee=(m'_1,\ldots,m'_{p'})$ determined by
$|\bfm|:=m_1+\cdots+m_p=m'_1+\cdots+m'_{p'}$ and
\begin{equation*}
\{1,2,\ldots,|\bfm|-1\}
=\Big\{ \sideset{}{_{i=1}^{j}}\sum m_i\Big\}_{j=1}^{p-1}
 \coprod \Big\{ \sideset{}{_{i=1}^{j}}\sum  m_i'\Big\}_{j=1}^{p'-1}.
\end{equation*}
For example, we have
$({1,1,2,1})^\vee=(3,2)\quad\text{and}\quad ({1,2,1,1})^\vee=(2,3).$

Setting $a_1=\cdots=a_{p-1}=0$ and $a_p=x$ in \eqref{Int-G} we can recover the identity in \cite[Thm. 2.7]{X2020}.
Replacing $m_{j}$ by $m_{j}-1$ for $j\ge 2$ we can restate it as
\begin{align}\label{IMP-2}
&\zeta_n^\star(\bfm^\vee;x)
-\sum\limits_{j=1}^p (-1)^{p-j}\zeta^\star_n(\ora\bfm_{\hskip-2pt j}^\vee){\Li}_{\ola\bfm_{\hskip-2pt j+1}}(1-x)\nonumber\\
=&n(-1)^{p}\int_x^1
\left(\frac{dt}{1-t}\right)^{m_p-1}\frac{dt}{t} \cdots \left(\frac{dt}{1-t}\right)^{m_2-1}\frac{dt}{t}\left(\frac{dt}{1-t}\right)^{m_1}\frac{dt}{t}t^{n-1}dt.
\end{align}

\section{Some Variants of Convoluted MZVs}\label{SRCVMZV}

\subsection{Kaneko-Yamamoto Type Convoluted MZVs}
Kaneko and Yamamoto initiated the study of the convoluted MZVs defined by
$$
\zeta(\bfk\circledast\bfl^\star)=\sum\limits_{n=1}^\infty \frac{\zeta_{n-1}(k_2,\ldots,k_r)\zeta^\star_n(l_2,\ldots,l_s)}{n^{k_1+l_1}}.$$
Their main result is a class of relations between these values stated in \cite[Thm. 4.1]{KanekoYa2018}. They further conjecture that any linear dependency of MZVs over $\Q$ can be deduced from these relations. In this section, we shall study relations
among some generalizations of these values. First, we consider a result involving MPLs.

\begin{thm} For any two compositions $\bfk=(k_1,\dotsc,k_r)$, $\bfm=(m_1,\dotsc,m_p)$, and $|x|\le 1$
\begin{align}\label{KYMZV1}
&\sum_{j=1}^p (-1)^{j-1} \Li_{\ola\bfm_{\hskip-2pt j}}\left(\sigma_px,\frac{\sigma_{p-1}}{\sigma_p},\ldots,\frac{\sigma_{j}}{\sigma_{j+1}}\right)\nonumber\\&\quad\quad\times\su \frac{\ze_{n-1}\left( \ora\bfk_{\hskip-2pt 2,r};\frac{\varepsilon_2}{\varepsilon_1},\ldots,\frac{\varepsilon_r}{\varepsilon_{r-1}}\right)
\ze^\star_n\left(\ora\bfm_{\hskip-2pt j-1};\frac{\sigma_1}{\sigma_2},\ldots,\frac{\sigma_{j-1}}{\sigma_{j}}\right)}{n^{k_1+1}}
\left(\frac{\varepsilon_1}{\sigma_1}\right)^n\nonumber\\
&+(-1)^p \su \frac{\ze_{n-1}\left( \ora\bfk_{\hskip-2pt 2,r};\frac{\varepsilon_2}{\varepsilon_1},\ldots,\frac{\varepsilon_r}{\varepsilon_{r-1}}\right)
\ze^\star_n\left(\bfm;\frac{\sigma_1}{\sigma_2},\ldots,\frac{\sigma_{p-1}}{\sigma_p},\sigma_p x\right)}{n^{k_1+1}}
\left(\frac{\varepsilon_1}{\sigma_1}\right)^n\nonumber\\
&=\Li_{\ola\bfm,k_1+1,\ora\bfk_{\hskip-2pt 2,r}}\left(\sigma_p x, \frac{\sigma_{p-1}}{\sigma_p},\ldots,\frac{\sigma_{1}}{\sigma_2},\frac{\varepsilon_1}{\sigma_1},\frac{\varepsilon_{2}}{\varepsilon_1},\ldots,\frac{\varepsilon_{r}}{\varepsilon_{r-1}}\right),
\end{align}
where $|\varepsilon_i|<1,|\sigma_{j}|<1\ (i=1,2,\ldots,r;j=1,2,\ldots,p)$ with $|\varepsilon_1/\sigma_1|<1$.
\end{thm}
\begin{proof}
Multiplying \eqref{FII2} by $n^{-k_1-1} \ze_{n-1}\left( \ora\bfk_{\hskip-2pt 2,r};\frac{\varepsilon_2}{\varepsilon_1},\ldots,\frac{\varepsilon_r}{\varepsilon_{r-1}}\right)$ and summing up, then applying \eqref{MPL-II-1}, we obtain the desired formula with an elementary calculation.
\end{proof}

Letting all $\varepsilon_i\rightarrow 1$ and $\sigma_{j}\rightarrow 1$ in \eqref{KYMZV1}, we can get the following corollary.

\begin{cor} For $\bfk=(k_1,\dotsc,k_r)$, $\bfm=(m_1,\dotsc,m_p)$, and $|x|\le 1$ with $(m_1,x)\ne(1,1)$,
\begin{multline}\label{KYMZV2}
\Li_{\ola\bfm,k_1+1,\ora\bfk_{\hskip-2pt 2,r}}(x)=(-1)^p \su \frac{\ze_{n-1}(\ora\bfk_{\hskip-2pt 2,r})\ze^\star_n(\bfm;x)}{n^{k_1+1}}
-\sum_{j=1}^p (-1)^{j}\Li_{\ola\bfm_{\hskip-2pt j}}(x)\ze\left(\bfk\circledast(1,\ora\bfm_{\hskip-2pt j-1})^\star\right).
\end{multline}
\end{cor}

In \cite[Thm. 3.1 and Cor. 3.2]{X2020}, we gave some explicit evaluations of the series of the form
\[\sum\limits_{n=1}^\infty \frac{\zeta_{n-1}(k_2,\ldots,k_r)\zeta^\star_n ((m_1,m_2+1,\ldots,m_p+1)^\vee;x)}{n^{k_1+1}}.\]
Here if $p=1$ then $(m_1,m_2+1,\ldots,m_p+1)^\vee:=(m_1)^\vee$.
In particular, we found the formula that for $\bfm=(m_1,\ldots,m_p)\in\N_0^p$ with $m_1\geq 1$ and $k\in\N$,
\begin{align}\label{KYMZV3}
&\sum\limits_{n=1}^\infty \frac{\zeta_{n-1}(\yi_{r-1})\zeta^\star_n ((m_1,m_2+1,\ldots,m_p+1)^\vee;x)}{n^{k+1}}\nonumber\\
&=(-1)^{p+k-1}\sum\limits_{|\bfj|+j_k=r,\atop j_k\geq 0}\sum\limits_{i_0+|\bfi|=j_k,\atop i_0 \geq 0} \frac{(-1)^{i_0}}{i_0!}\left\{\prod\limits_{l=1}^p \binom{m_l+i_l}{i_l}\right\}\log^{i_0}(1-x)\Li_{\ola{\bfm+\bfi+\bfone} ,\bfj+\bfone} (1-x)\nonumber\\
&\quad+(-1)^p\sum\limits_{j=0}^{k-2} (-1)^j \zeta(k-j,\yi_{r-1})\Li_{\ola{\bfm+\bfone},\yi_{j}}(1-x)\nonumber\\
&\quad + \sum\limits_{j=1}^{p} (-1)^{p-j} \zeta\Big((k,\yi_{r-1})\circledast \big(1,(m_1,m_2+1,\ldots,m_{j}+1)^\vee\big)^\star\Big)\Li_{\ola{\bfm+\bfone}_{j+1}}(1-x),
\end{align}
where $|x|<1$, $\bfi:=(i_1,\ldots,i_p)\in\N_0^p$ and $\bfj:=(j_1,\ldots,j_{k-1})\in\N_0^{k-1}$.

On the other hand, for any composition $\bfk$, Kaneko and Tsumura \cite{KT2018} proved that
\begin{align}\label{MPL-R1}
\Li_{\bfk}(1-x)=\sum\limits_{\bfk',j\geq 0,|\bfk'|+j\leq |\bfk|} c_\bfk(\bfk';j)\Li_{\yi_{j}}(1-x)\Li_{\bfk'}(x),
\end{align}
where $c_\bfk(\bfk';j)$ is a $\mathbb{Q}$-linear combination of multiple zeta values of weight $|\bfk|- |\bfk'|-j$. We understand $\Li_{\emptyset}(x)=1$ and $|\emptyset|=0$ for the empty index $\emptyset$, and the constant $1$ is regarded as a multiple zeta value of weight $0$. However, the explicit formula of the coefficient $c_\bfk(\bfk';j)$ was not found. As an example, Arakawa-Kaneko \cite[Thm. 8]{AM1999} showed that for $r\ge 0$ and $k\ge 2$
\begin{equation}\label{MPL-Rs}
\Li_{k,\yi_r}(x)=\sum\limits_{j=0}^{k-2} (-1)^j \zeta(k-j,\yi_r)\Li_{\yi_{j}}(1-x)
-(-1)^{k} \sum_{\substack{ |\bfi|+\ell=r+k\\  \bfi\in\N^{k-1},\, \ell \geq 0}} \Li_{\yi_\ell}(x)\Li_\bfi(1-x).
\end{equation}
We can also apply \eqref{KYMZV2} and \eqref{KYMZV3} to find some explicit relations of \eqref{MPL-R1} which generalizes \eqref{MPL-Rs}.

\begin{thm} For all $p,r,k\in\N_0$, $k\ge 1$ and $|1-x|<1$,
\begin{multline} \label{MPL-R2}
(-1)^p\Li_{\yi_p,k+1,\yi_r}(1-x)=\sum_{j=0}^{p} \ze((k,\yi_r)\circledast(\yi_{p+1-j})^\star)\frac{\log^{j}(x)}{j!}
    -\sum_{j=0}^{k-2}(-1)^j\ze(k-j,\yi_r)\Li_{p+1,\yi_{j}}(x)\\
+(-1)^{k}\sum_{\substack{|\bfi|+j+\ell=r+k \\  \bfi\in\N^{k-1}, j,\ell\geq 0}}
    \frac{(-1)^{\ell}}{\ell!}\binom{p+j}{j}\log^{\ell}(x)\Li_{p+j+1,\bfi}(x).
\end{multline}
\end{thm}
\begin{re} This is a special case of \eqref{MPL-R1} since $\ze((k,\yi_{r})\circledast(\yi_{p+1-j})^\star)$
can be expressed as a $\Q$-linear combinations of MZVs for all $j$ by \cite[Thm. 4.1]{KanekoYa2018} while
\begin{equation}\label{equ:itLi1j}
\Li_{\yi_{j}}(x)=\frac{(-1)^j}{j!}\log^j(1-x).
\end{equation}
\end{re}
\begin{proof} In \eqref{KYMZV2}, replacing $x$ by $1-x$, letting $k_1=k$ and $k_i=m_{j}=1$ for all $2\le i\le r$ and $1\le j\le p$, and
replacing $r$ by $r+1$, we have
\begin{align*}
\Li_{\yi_p,k+1,\yi_{r}}(1-x)&=\sum_{j=1}^p (-1)^{j-1} \Li_{\yi_{p-j+1}}(1-x)\ze((k,\yi_{r})\circledast(\yi_{j})^\star)\\
&\quad+(-1)^p\su \frac{\ze_{n-1}(\yi_{r})\ze_n^\star(\yi_p;1-x)}{n^{k+1}}.
\end{align*}
Then, setting $p=1$ and $m_1=p$ so that $(m_1)^\vee=(\yi_p)$ in \eqref{KYMZV3} and
replacing $x$ by $1-x$ we can readily deduce the desired result by \eqref{equ:itLi1j}.
\end{proof}

If setting $k=1$ in \eqref{MPL-R2} we obtain \cite[Eq. (2.12)]{X2020}. Moreover, equation \eqref{MPL-R2} also holds for $p=0$ which implies \eqref{MPL-Rs} with the substitution $k\to k-1$.

\begin{thm} For any $m,k\in\N$, $r\in\N_0$, $|x|<1$ and $|1-x|\le 1$ with $(m,x)\neq (1,0)$, we have
\begin{multline}\label{MPL-R3}
\Li_{m,k+1,\yi_{r}}(1-x)=\Li_m(1-x)\ze(k+1,\yi_{r})
-\sum_{j=0}^{k-2}(-1)^{m+j}\ze(k-j,\yi_{r})\Li_{\yi_{m-1},2,\yi_{j}}(x)\\
+(-1)^{k+m} \sum_{|\bfi|=r, \bfi\in\N_0^{m+k}} \Li_{\bfi+\bfone+\bfe_m}(x)
 -\sum_{j=1}^m \ze\big((k,\yi_{r})\circledast(1,j)^\star\big)\frac{\log^{m-j}(1-x)}{(m-j)!},
\end{multline}
where $\bfe_m=(0_m,1,0_{k-1})$.
\end{thm}
\begin{proof}
The proof is similar to the proof of \eqref{MPL-R2}. Replacing $x$ by $1-x$ and $r$ by $r+1$, setting $p=1$ with $m_1=m$,
$k_1=k$, $k_2=\cdots=k_{r+1}=1$  in \eqref{KYMZV2}, we can prove \eqref{MPL-R3} by applying \eqref{KYMZV3}.
\end{proof}

\subsection{$t$-Variants of Convoluted MZVs}
In \cite{H2019}, Hoffman introduced and studied odd variants of MZVs and MZSVs, which are
defined for an admissible composition $\bfk=(k_1,k_2,\ldots,k_r)$ by
\begin{align}\label{defn-mtvs}
t(\bfk):=\sum_{n_1>n_2>\cdots>n_r>0} \prod_{j=1}^r \frac{1}{(2n_{j}-1)^{k_{j}}}\quad
\text{and}\quad
t^\star(\bfk):=\sum_{n_1\geq n_2\geq \cdots\geq n_r>0} \prod_{j=1}^r \frac{1}{(2n_{j}-1)^{k_{j}}},
\end{align}
and are called a \emph{multiple $t$-value} and \emph{multiple $t$-star value}, respectively.

Then, similar to multiple harmonic sums and multiple harmonic star sums, for a composition $\bfk=(k_1,\ldots,k_r)$ and positive integer $n$, we define the \emph{multiple t-harmonic sums} and \emph{multiple t-harmonic star sums} respectively by
\begin{align*}
t_n(\bfk):=\sum_{n\geq n_1>n_2>\cdots>n_r>0} \prod_{j=1}^r \frac{1}{(2n_{j}-1)^{k_{j}}}\quad
\text{and}\quad
t^\star_n(\bfk):=\sum_{n\geq n_1\geq n_2\geq \cdots\geq n_r>0} \prod_{j=1}^r \frac{1}{(2n_{j}-1)^{k_{j}}}.
\end{align*}
Therefore, we define the \emph{$t$-variant of convoluted MZVs} by
\begin{align*}
t(\bfk\circledast\bfl^\star)&=\sum\limits_{0<m_r<\cdots<m_1=n_1\geq \cdots \geq n_s>0} \prod_{j=1}^r\frac{1}{(2m_{j}-1)^{k_{j}}} \prod_{j=1}^r\frac{1}{(2n_{j}-1)^{l_{j}}}
=\sum\limits_{n=1}^\infty \frac{t_{n-1}(\ora\bfk_{\hskip-2pt 2,r})t^\star_n(\ora\bfl_{\hskip-2pt 2,s})}{(2n-1)^{k_1+l_1}}.
\end{align*}
We also define the \emph{multiple t-polylogarithm function} for any composition $\bfk=(k_1,\ldots,k_r)$ by
\begin{align*}
{\ti}_{\bfk}(x)&:=\sum_{n_1>n_2>\cdots>n_r>0} \frac{x^{2n_1-1}}{(2n_1-1)^{k_1}(2n_2-1)^{k_2}\ldots (2n_r-1)^{k_r}}\quad \nonumber\\
&=\int_0^x \frac{dt}{1-t^2}\left(\frac{dt}{t}\right)^{k_r-1} \frac{tdt}{1-t^2}\left(\frac{dt}{t}\right)^{k_{r-1}-1}\cdots \frac{tdt}{1-t^2}\left(\frac{dt}{t}\right)^{k_1-1},
\end{align*}
where $|x|\le 1$ with $(k_1,x)\ne (1,1)$. Clearly, $\ti_{\bfk}(1)=t(\bfk)$ with $k_1\geq 2$.

\begin{thm}\label{thm-ItIt}
For any $\bfm=(m_1,\ldots,m_p)\in\N^p$, $n\in\N$ and $|x|<1$, we have
\begin{align}\label{FIIt1}
&2n\int_0^x t^{2n-1}dt \frac{dt}{1-t^2} \left(\frac{dt}{t}\right)^{m_1-1}\frac{tdt}{1-t^2} \left(\frac{dt}{t}\right)^{m_2-1}\cdots \frac{tdt}{1-t^2}\left(\frac{dt}{t}\right)^{m_p-1}\nonumber\\
&=(2n-1)\int_0^x t^{2n-2}dt \frac{tdt}{1-t^2} \left(\frac{dt}{t}\right)^{m_1-1}\frac{tdt}{1-t^2} \left(\frac{dt}{t}\right)^{m_2-1}\cdots \frac{tdt}{1-t^2}\left(\frac{dt}{t}\right)^{m_p-1}\nonumber\\
&=\sum_{j=1}^{p}(-1)^{j-1} t^\star_n(\ora\bfm_{\hskip-2pt j-1})\ti_{\ola\bfm_{\hskip-2pt j}}(x)+(-1)^pt^\star_n(\bfm;x),
\end{align}
where
\begin{align}\label{equ:t-mhss}
t^\star_n(\bfk;x):=\sum_{n\geq n_1\geq n_2\geq \cdots \geq n_r\geq 1} \frac{x^{2n_r-1}}{(2n_1-1)^{k_1}(2n_2-1)^{k_2}\ldots(2n_r-1)^{k_r}}.
\end{align}
\end{thm}
\begin{proof}
The proof is similar to that of \eqref{FII1} which is left to the
interested reader.
\end{proof}

\begin{thm}
For compositions $\bfk=(k_1,\dotsc,k_r)$ and $\bfm=(m_1,\dotsc,m_p)$, $|x|\le 1$ and $(m_p,x)\neq (1,1)$, we have
\begin{multline} \label{t-KYMZVx}
\ti_{\ola\bfm,k_1+1,\ora\bfk_{\hskip-2pt 2,r}}(x)=(-1)^p\su \frac{t_{n-1}(\ora\bfk_{\hskip-2pt 2,r})t^\star_n(\bfm;x)}{(2n-1)^{k_1+1}}
-\sum_{j=1}^p (-1)^{j} \ti_{\ola\bfm_{\hskip-2pt j}}(x)t(\bfk\circledast(1,\ora\bfm_{\hskip-2pt j-1})^\star).
\end{multline}
\end{thm}
\begin{proof} This is similar to the proof of \eqref{KYMZV1}.
Multiplying \eqref{FII2} by $(2n-1)^{k_1+1}  t_{n-1}(\ora\bfk_{\hskip-2pt 2,r}) $ and summing up
we get \eqref{t-KYMZVx} immediately.
\end{proof}

In particular, assuming $m_p\geq 2$ and setting $x=1$ in \eqref{t-KYMZVx} we get
\begin{equation}\label{t-KYMZVxx}
\sum_{j=1}^{p+1} (-1)^{j-1} t(\ola\bfm_{\hskip-2pt j})t\big(\bfk\circledast(1,\ora\bfm_{\hskip-2pt j-1})^\star\big)
=t(\ola\bfm,k_1+1,\ora\bfk_{\hskip-2pt 2,r}).
\end{equation}

\subsection{$M$-Variants of Convoluted MZVs}
In \cite{XuZhao2020a}, we define the multiple mixed values or multiple $M$-values (MMVs) for an admissible composition $\bfk=(k_1,\ldots,k_r)$ and $\bfeps=(\varepsilon_1,\ldots,\varepsilon_r)\in\{\pm1\}^r$ by
\begin{align}\label{defn-mmvs}
M(\bfk;\bfeps)&:=\sum_{n_1>n_2>\cdots>n_r>0} \prod_{j=1}^r \frac{1+\varepsilon_{j}(-1)^{n_{j}}}{n_{j}^{k_{j}}}
=\int_0^1 w_{\varepsilon_r}w_0^{k_r-1}\cdots w_{\varepsilon_2\varepsilon_3}w_0^{k_2-1}w_{\varepsilon_1\varepsilon_2}w_0^{k_1-1},
\end{align}
where
\begin{equation*}
w_0(t):=\frac{dt}{t},\quad w_{-1}:=\frac{2dt}{1-t^2},\quad w_1:=\frac{2tdt}{1-t^2}.
\end{equation*}

We define the \emph{multiple $M$-polylogarithm function} for any composition $\bfk=(k_1,k_2,\ldots,k_r)$, $\bfeps=(\eps_1,\ldots,\eps_r)\in\{\pm 1\}^r$ and $|x|\le 1$ with $(k_1,x)\neq (1,1)$ by
\begin{align}\label{defn-mmpls}
\Mi_{\bfk}(\bfeps;x)&:=\sum_{n_1>n_2>\cdots>n_r>0}  x^{n_1}\prod_{j=1}^r \frac{1+\varepsilon_{j}(-1)^{n_{j}}}{n_{j}^{k_{j}}}  \nonumber\\
&=\int_0^x w_{\varepsilon_r}w_0^{k_r-1} w_{\varepsilon_r\varepsilon_{r-1}}w_0^{k_{r-1}-1}\cdots w_{\varepsilon_2\varepsilon_1}w_0^{k_1-1}.
\end{align}
Clearly, $\Mi_{\bfk}(\bfeps;1)=M(\bfk;\bfeps)$ for all $k_1\ge 2$.

Similar to the multiple harmonic sums and multiple harmonic star sums, for a composition $\bfk=(k_1,\ldots,k_r)$, $\bfeps=(\eps_1,\ldots,\eps_r)\in\{\pm 1\}^r$ and positive integer $n$, we may define the \emph{multiple $M$-harmonic sums} and \emph{multiple $M$-harmonic star sums} respectively by
\begin{equation*}
M_n(\bfk;\bfeps):=\sum\limits_{n\geq n_1> \dotsm > n_r>0} \prod_{j=1}^r \frac{1+\eps_{j}(-1)^{n_{j}}}{n_{j}^{k_{j}}} \quad
\text{and}\quad
M_n^\star(\bfk;\bfeps):=\sum\limits_{n\geq n_1\geq  \dotsm \geq n_r\geq 1} \prod_{j=1}^r\frac{1+\eps_{j}(-1)^{n_{j}}}{n_{j}^{k_{j}}} ,
\end{equation*}
where $M_n(\emptyset;\emptyset)=M_n^\star(\emptyset;\emptyset):=1$.

\begin{defn}
For any two compositions of positive integers $\bfk=(k_1,\dotsc,k_r)$ and $\bfl=(l_1,\dotsc,l_p)$, and $\bfet:=(\eta_1,\dotsc,\eta_r)\in\{\pm1,0\}\times\{\pm1\}^{r-1}$ and $\bfeps:=(\varepsilon_1,\dotsc,\varepsilon_p)\in\{\pm1,0\}\times\{\pm1\}^{p-1}$ with $(\eta_1,\eps_1)\neq (0,0)$, define the \emph{convoluted multiple mixed values} by
\begin{align}\label{equ:convMMVs}
M((\bfk;\bfet)\circledast(\bfl;\bfeps)^\star)
&=\sum\limits_{n=1}^\infty \frac{M_{n-1}(\ora\bfk_{\hskip-2pt 2,r};\ora\bfet_{\hskip-2pt 2,r})M^\star_n(\ora\bfl_{\hskip-2pt \hskip-3pt 2,p};\ora\bfeps_{\hskip-2pt  2,p})}{n^{k_1+l_1}} \frac{(1+\varepsilon_1(-1)^n)(1+\eta_1(-1)^{n})}{2}.
\end{align}
\end{defn}

\begin{re} \label{rem:ConvMMV}
The factors at the end of \eqref{equ:convMMVs} imply that \eqref{equ:convMMVs} vanishes if
$\varepsilon_1\ne \eta_1$.
\end{re}

We define three maps: for any $\bfeps=(\eps_1,\ldots,\eps_r)\in \{\pm 1\}^r$,
\begin{align*}
&\bfp(\bfeps)\equiv\bfp(\eps_1,\eps_2,\ldots,\eps_r)=(\eps_1\eps_2\cdots\eps_r,\ldots,\eps_{r-1}\eps_r,\eps_r),\\
&\bfq(\bfeps)\equiv\bfq(\eps_1,\eps_2,\ldots,\eps_r)=(\eps_1\eps_2\cdots\eps_r,\ldots,\eps_{1}\eps_2,\eps_1),\\
&\bfr(\bfeps)\equiv\bfr(\eps_1,\eps_2,\ldots,\eps_r)=(\eps_1,\eps_{1}\eps_2,\ldots,\eps_1\eps_2\cdots\eps_r).
\end{align*}
Set $\bfp(\emptyset)=\bfq(\emptyset)=\bfr(\emptyset):=0$, and
$a\bfeps=(a\eps_1,\dotsc,a\eps_r)$ for any real number $a$.

\begin{thm}\label{thm-ItItmmv}
Let $n=n_0\in\N$,  $\bfm=(m_1,m_2,\ldots,m_p)\in\N^p$ and $(\eps_1,\ldots,\eps_p)\in\{\pm 1\}^p$. Then
\begin{align}\label{equ:Intx-MMVs}
&n\int_0^x t^{n-1} dt w_{\eps_1}w_0^{m_1-1}w_{\eps_2}w_0^{m_2-1}\cdots w_{\eps_p}^{m_p-1}\nonumber\\
=&(-1)^p \sum_{n\geq n_1 \geq \cdots \geq n_p\geq 1} x^{n_p}\prod_{j=1}^p \frac{1+\eps_{j}(-1)^{n_{j-1}+n_{j}}}{n_{j}^{m_{j}}}\nonumber\\
+&\frac1{2} \sum_{j=1}^p (-1)^{j-1}\Mi_{\ola\bfm_{\hskip-2pt j}}\Big(-\bfq(\ora\bfeps_{\hskip-3pt j+1,p}),-1;x\Big) (1-\eps_1\eps_2\cdots\eps_{j}(-1)^n) M_n^\star\Big(\ora\bfm_{\hskip-2pt j-1};-\bfp(\ora\bfeps_{\hskip-2pt 2,j})\Big)\nonumber\\
+&\frac1{2} \sum_{j=1}^p (-1)^{j-1}\Mi_{\ola\bfm_{\hskip-2pt j}}\Big(\bfq(\ora\bfeps_{\hskip-3pt j+1,p}),1;x\Big) (1+\eps_1\eps_2\cdots\eps_{j}(-1)^n) M_n^\star\Big(\ora\bfm_{\hskip-2pt j-1};\bfp(\ora\bfeps_{\hskip-2pt 2,j})\Big).
\end{align}
\end{thm}
\begin{proof}
The proof of \eqref{equ:Intx-MMVs} is similar to the proof of \eqref{FII1}.
First observe that
\begin{equation*}
t^n w_{\eps}=\frac{1-\eps(-1)^n}{2}w_{-1}+\frac{1+\eps(-1)^n}{2}w_1-\sum_{m=1}^n (1+\eps(-1)^{n+m})t^{m-1}dt.
\end{equation*}
We also note that $1+\eps(-1)^{n}\ne 0$ if and only if $\eps=(-1)^{n}$ which implies that
\begin{align*}
&(1+\eps_1(-1)^{n+n_1})\cdots (1+\eps_{j-1}(-1)^{n_{j-2}+n_{j-1}})(1+\eps_{j}(-1)^{n_{j-1}})\\
&=(1+\eps_1\eps_2\cdots \eps_{j}(-1)^{n})\cdots (1+\eps_{j-1} \eps_{j}(-1)^{n_{j-2}})(1+\eps_{j}(-1)^{n_{j-1}}).
\end{align*}
The theorem now follows from \eqref{defn-mmvs} and
we leave the details to the interested reader.
\end{proof}

\begin{thm}\label{thm-mmvs}  For any two compositions $\bfk=(k_1,k_2,\ldots,k_r)$ and $\bfm=(m_1,m_2,\ldots,m_p)$, and $\bfet=(\eta_1,\ldots,\eta_r)\in\{\pm 1\}^r, \bfeps=(\eps_1,\ldots,\eps_p)\in\{\pm 1\}^p$,
\begin{align}\label{equ:Intx-KYMMVs}
&M_{m_p,\ldots,m_1,k_1+1,k_2,\ldots,k_r}(\eta_1\eta_2\cdots\eta_r\bfq(\bfeps),\bfp(\bfet);x)\nonumber\\
=&(-1)^p \su \frac{M_{n-1}\big(\ora\bfk_{\hskip-2pt 2,r};\bfp(\ora\bfet_{\hskip-2pt 2,r})\big)}{n^{k_1+1}}(1+\eta_1\cdots \eta_r(-1)^{n})
\sum_{n\geq n_1 \geq \cdots \geq n_p\geq 1} x^{n_p}\prod_{j=1}^p \frac{1+\eps_{j}(-1)^{n_{j-1}+n_{j}}}{n_{j}^{m_{j}}}\nonumber \\
&-\sum_{j=1}^p(-1)^{j} \Mi_{\ola{\bfm}_{j}}\Big(-\bfq(\ora\bfeps_{\hskip-3pt j+1,p}),-1;x\Big)M\Big(\big(\bfk;\bfp(\bfet)\big)\circledast\big(1,\ora\bfm_{\hskip-2pt j-1};-\bfp(\ora\bfeps_{\hskip-3pt j})\big)^\star\Big)\nonumber\\
&-\sum_{j=1}^p(-1)^{j} \Mi_{\ola{\bfm}_{j}}\Big(\bfq(\ora\bfeps_{\hskip-3pt j+1,p}),1;x\Big)M\Big(\big(\bfk;\bfp(\bfet)\big)\circledast\big(1,\ora\bfm_{\hskip-2pt j-1};\bfp(\ora\bfeps_{\hskip-3pt j})\big)^\star\Big).
\end{align}
\end{thm}
\begin{proof}
Multiplying \eqref{equ:Intx-MMVs} by $n^{-k_1-1}M_{n-1}\big(\ora\bfk_{\hskip-2pt 2,r};\bfp(\ora\bfet_{\hskip-2pt 2,r})\big) (1+\eta_1\cdots \eta_r(-1)^n)$ and summing up we can obtain the evaluation by \eqref{defn-mmvs}.
\end{proof}

\begin{cor}\label{cor-mmvs}  For any two compositions $\bfk=(k_1,k_2,\ldots,k_r)$ and $\bfm=(m_1,m_2,\ldots,m_p)$, and $\bfet=(\eta_1,\ldots,\eta_r)\in\{\pm 1\}^r, \bfeps=(\eps_1,\ldots,\eps_p)\in\{\pm 1\}^p$,
\begin{align}\label{equ:Intx-KYMMVsCorollary}
&M(m_p,\ldots,m_1,k_1+1,k_2,\ldots,k_r;\eta_1\eta_2\cdots\eta_r\bfq(\bfeps),\bfp(\bfet))\nonumber\\
&=\sum_{j=1}^p(-1)^{j-1} M(\ola{\bfm}_{j};\bfq(\ora\bfeps_{\hskip-3pt j+1,p}),1)M\Big(\big(\bfk;\bfp(\bfet)\big)\circledast\big(1,\ora\bfm_{\hskip-2pt j-1};\bfp(\ora\bfeps_{\hskip-3pt j})\big)^\star\Big)\nonumber\\
&\quad+\sum_{j=1}^p(-1)^{j-1} M(\ola{\bfm}_{j};-\bfq(\ora\bfeps_{\hskip-3pt j+1,p}),-1)M\Big(\big(\bfk;\bfp(\bfet)\big)\circledast\big(1,\ora\bfm_{\hskip-2pt j-1};-\bfp(\ora\bfeps_{\hskip-3pt j})\big)^\star\Big)\nonumber\\
&\quad+2(-1)^p M\Big(\big({\bfk};\bfp(\bfet)\big)\circledast \big(1,{\bfm};0,\eta_1\cdots \eta_r\bfr(\bfeps)\big)^\star\Big).
\end{align}
\end{cor}
\begin{proof}
Letting $x=1$ in \eqref{equ:Intx-KYMMVs} and noting the fact that
\begin{align*}
&(1+\eta_1\cdots \eta_r(-1)^n)(1+\eps_1(-1)^{n+n_1})\cdots (1+\eps_p(-1)^{n_{p-1}+n_p})\\
&=(1+\eta_1\cdots \eta_r(-1)^n)(1+\eta_1\cdots \eta_r\eps_1(-1)^{n_1})\cdots (1+\eta_1\cdots \eta_r\eps_1\cdots \eps_p(-1)^{n_p}),
\end{align*}
we may obtain the desired result by a straight-forward calculation,.
\end{proof}

Setting $\eta_r=-1$ and $\eta_{j}=\eps_i=1\ (1\le j<r,1\le i\le p)$ in the corollary we can recover formula \eqref{t-KYMZVxx} once again
since (i) the second factor in the first sum on the right-hand side of \eqref{equ:Intx-KYMMVsCorollary} is always zero according to Remark~ \ref{rem:ConvMMV}, and (ii) we can replace 0 in the last term of \eqref{equ:Intx-KYMMVsCorollary} by $-1$ and remove the
factor 2 in the front by Remark~ \ref{rem:ConvMMV} again.

\section{Parametric Multiple Polylogarithms}\label{SFPMPLS}

In this section, we define the parametric MPL (star) function and parametric multiple harmonic (star) sum, and establish some new identities. In particular, we extend the results of Sakugawa-Seki \cite[Thms. 2.10 and 2.13]{SS2016}.

\begin{defn}
For any complex parameter $a$, $\bfk=(k_1,\ldots,k_r)\in\N^r$, and complex variables $\bfx=(x_1,\dotsc,x_r)\in\gD^r$, we define the \emph{parametric multiple harmonic sum} and \emph{parametric multiple harmonic star sum} with $r$-variable by
\begin{align*}
\zeta_n(\bfk;\bfx;a):=\sum\limits_{n\geq n_1> \cdots > n_r\geq 1} \prod_{j=1}^r \frac{x_{j}^{n_{j}+a}}{(n_{j}+a)^{k_{j}}}\quad
\text{and}\quad
\zeta^\star_n(\bfk;\bfx;a):=\sum\limits_{n\geq n_1\geq \cdots \geq n_r\geq 1} \prod_{j=1}^r \frac{x_{j}^{n_{j}+a}}{(n_{j}+a)^{k_{j}}},
\end{align*}
respectively, where if $n<k$ then ${\zeta_n}(\bfk;\bfx;a):=0$ and $\ze_n(\emptyset;\emptyset;a)=\ze^\star_n(\emptyset;\emptyset;a):=1$.
For $(k_1,x_1)\neq (1,1)$, we define the \emph{parametric multiple polylogarithm function} and \emph{parametric multiple polylogarithm star function} with $r$-variable by
\begin{align*}
\Li_\bfk(\bfx;a):=\lim_{n\to\infty} \zeta_n(\bfk;\bfx;a), \quad
\text{and} \quad
\Li^\star_\bfk(\bfx;a):=\lim_{n\to\infty}\zeta^\star_n(\bfk;\bfx;a).
\end{align*}
\end{defn}

\begin{lem}\label{lem-ss2016} \emph{(\cite[Thm. 2.11]{SS2016})} For any composition ${\bfk}=(k_1,\ldots,k_r)$ and $\bfx=(x_1,\dotsc,x_r)\in\gD^r$, we have
\begin{align}\label{eq-ss2016}
\sum_{j=0}^r (-1)^j \ze_n(\ora\bfk_{\hskip-2pt j};\ora\bfx_{\hskip-2pt j})\ze^\star_n(\ola\bfk_{\hskip-2pt j+1}; \ola\bfx_{\hskip-2pt j+1})=0,
\end{align}
where $\ze_n(\emptyset;\emptyset)=\ze^\star_n(\emptyset;\emptyset):=1$.
\end{lem}

\begin{thm}\label{thm-PMPLs1}
For any composition ${\bfk}=(k_1,\ldots,k_r)$, $\bfx=(x_1,\dotsc,x_r)\in\gD^r$ and any $l\in\N_0,n\in\N$, we have
\begin{align}\label{eq-PMPLs1}
&\sum_{n\geq n_1>\cdots>n_r>0} \prod_{j=1}^r \frac{x_{j}^{n_{j}+l}}{(n_{j}+l)^{k_{j}}}
=(-1)^r \sum_{j=0}^r (-1)^j\ze_{n+l}(\ora\bfk_{\hskip-2pt j};\ora\bfx_{\hskip-2pt j}) \ze^\star_l(\ola\bfk_{\hskip-2pt j+1}; \ola\bfx_{\hskip-2pt j+1}).
\end{align}
\end{thm}
\begin{proof}
We proceed by induction on the depth $r$. When the depth is $1$, the left hand side is $\sum_{m=1}^n\frac{x_1^{m+l}}{(m+l)^{k_1}}$ and the right-hand side is $\ze_{n+l}(k_1;x_1)-\ze^\star_l(k_1;x_1)=\sum_{m=l+1}^{n+l}\frac{x_1^m}{m^{k_1}}=\sum_{m=1}^n\frac{x_1^{m+l}}{(m+l)^{k_1}}$. Hence the case $r=1$ is proved.

Assume $r\ge 2$ and the theorem holds for lower depths. By inductive assumption we have
\begin{align*}
\text{LHS of \eqref{eq-PMPLs1}}
&=\sum_{n_1=1}^n \frac{x^{n_1+l}}{(n_1+l)^{k_1}}(-1)^{r-1} \sum_{j=0}^{r-1} (-1)^{j} \ze_{n_1+l-1}(\ora\bfk_{\hskip-2pt  2,j+1};\ora\bfx_{\hskip-2pt 2,j+1})\ze^\star_l(\ola\bfk_{\hskip-2pt j+2}; \ola\bfx_{\hskip-2pt j+2})\\
&=(-1)^{r} \sum_{j=1}^{r}  (-1)^j \left(\sum_{n_1=1}^n \frac{x^{n_1+l}}{(n_1+l)^{k_1}}\ze_{n_1+l-1}(\ora\bfk_{\hskip-2pt  2,j};\ora\bfx_{\hskip-2pt 2,j})\right)\ze^\star_l(\ola\bfk_{\hskip-2pt j+1}; \ola\bfx_{\hskip-2pt j+1})\\
&=(-1)^{r} \sum_{j=1}^{r} (-1)^j \left(\ze_{n+l}(\ora\bfk_{\hskip-2pt j};\ora\bfx_{\hskip-2pt j})-\ze_{l}(\ora\bfk_{\hskip-2pt j};\ora\bfx_{\hskip-2pt j})\right)
\ze^\star_l(\ola\bfk_{\hskip-2pt j+1}; \ola\bfx_{\hskip-2pt j+1})\\
&=(-1)^{r} \sum_{j=0}^{r} (-1)^j \ze_{n+l}(\ora\bfk_{\hskip-2pt j};\ora\bfx_{\hskip-2pt j})\ze^\star_l(\ola\bfk_{\hskip-2pt j+1}; \ola\bfx_{\hskip-2pt j+1}),
\end{align*}
where we used the Lemma \ref{lem-ss2016} in the last step. This completes the proof of Theorem \ref{thm-PMPLs1}.
\end{proof}

\begin{thm}\label{thm-PMPLs2}
For any composition ${\bfk}=(k_1,\ldots,k_r)$, $\bfx=(x_1,\dotsc,x_r)\in\gD^r$ and any $l\in\N_0,n\in\N$, we have
\begin{align}\label{eq-PMPLs2}
&\sum_{n\geq n_1\geq n_2\geq \cdots\geq n_r>0}  \prod_{j=1}^r \frac{x_{j}^{n_{j}+l}}{(n_{j}+l)^{k_{j}}}
=(-1)^r \sum_{j=0}^r (-1)^j\ze^\star_{n+l}(\ora\bfk_{\hskip-2pt j};\ora\bfx_{\hskip-2pt j}) \ze_l(\ola\bfk_{\hskip-2pt j+1}; \ola\bfx_{\hskip-2pt j+1}).
\end{align}
\end{thm}
\begin{proof}
The proof is completely similar to the proof of Theorem \ref{thm-PMPLs1} and is thus omitted.
\end{proof}

Observe that
\begin{align}
&n\int_0^x t^{n-1}dt \frac{dt}{1-\sigma_1t} \left(\frac{dt}{t}\right)^{m_1-1}\frac{dt}{1-\sigma_2t} \left(\frac{dt}{t}\right)^{m_2-1}\cdots \frac{dt}{1-\sigma_pt}\left(\frac{dt}{t}\right)^{m_p-1}\nonumber\\
&=\frac{\Li_{m_p,m_{p-1},\ldots,m_1}\left(\sigma_px,\frac{\sigma_{p-1}}{\sigma_p},\ldots,\frac{\sigma_1}{\sigma_{2}};n \right)}{\sigma_1^{n+1}\sigma_2\cdots \sigma_p}.
\end{align}
Hence,taking $n\rightarrow \infty$ in \eqref{eq-PMPLs1}, $(k_1,k_2,\ldots,k_r)\rightarrow (m_p,m_{p-1},\ldots,m_1)$ and $(x_1,x_2,\ldots,x_r)\rightarrow \left(\sigma_px,\frac{\sigma_{p-1}}{\sigma_p},\ldots,\frac{\sigma_1}{\sigma_{2}} \right)$ then replacing $l$ by $n$, we obtain formula \eqref{FII2}.

The next result is a parametric generalization of Lemma \ref{lem-ss2016}.
\begin{thm}\label{thm-pmhns}
For any complex parameter $a$, $\bfk=(k_1,\ldots,k_r)\in\N^r$, and complex variables $\bfx=(x_1,\dotsc,x_r)\in\gD^r$, we have
\begin{align}\label{eq-v2016}
\sum_{j=0}^r (-1)^j \ze_n(\ora\bfk_{\hskip-2pt j};\ora\bfx_{\hskip-2pt j};a)\ze^\star_n(\ola\bfk_{\hskip-2pt j+1}; \ola\bfx_{\hskip-2pt j+1};a)=0,
\end{align}
where $\ze_n(\emptyset;\emptyset;a)=\ze^\star_n(\emptyset;\emptyset;a):=1$.
\end{thm}
\begin{proof}
We proceed with induction on $r$. The case $r=1$ is trivial and the case $r=2$ is straight-forward.
Assuming $r\ge 3$, we have
\begin{align*}
&\sum_{j=0}^r (-1)^j \ze_n(\ora\bfk_{\hskip-2pt j};\ora\bfx_{\hskip-2pt j};a)\ze^{\star}_n(\ola\bfk_{\hskip-2pt j+1}; \ola\bfx_{\hskip-2pt j+1};a)\\
=&\ze^{\star}_n(\ola\bfk;\ola\bfx;a)+(-1)^r \ze_n(\bfk;\bfx;a)\\
+&\sum_{j=1}^{r-1} (-1)^j \sum_{n_1=1}^{n} \frac{x_1^{n_1+a}}{(n_1+a)^{k_1}} \ze_{n_1-1}(\ora\bfk_{\hskip-2pt 2,j};\ora\bfx_{\hskip-2pt 2,j};a)
\sum_{n_r=1}^{n} \frac{x_r^{n_r+a}}{(n_r+a)^{k_r}} \ze^{\star}_{n_r}(\ola\bfk_{\hskip-2pt j+1,r-1}; \ola\bfx_{\hskip-2pt j+1,r-1};a)\\
=&\ze^{\star}_n(\ola{\bfk};\ola{\bfx};a)+(-1)^r \ze_n(\bfk;\bfx;a)\\
+&
 \sum_{j=1}^{r-1} (-1)^j\sum_{n\ge n_r>n_1>0}
\frac{x_1^{n_1+a}}{(n_1+a)^{k_1}} \ze_{n_1-1}(\ora\bfk_{\hskip-2pt 2,j};\ora\bfx_{\hskip-2pt 2,j};a)
 \frac{x_r^{n_r+a}}{(n_r+a)^{k_r}} \ze^{\star}_{n_r}(\ola\bfk_{\hskip-2pt j+1,r-1}; \ola\bfx_{\hskip-2pt j+1,r-1};a)\\
+&
 \sum_{j=1}^{r-1} (-1)^j \sum_{n\ge n_1\ge n_r>0}
\frac{x_1^{n_1+a}}{(n_1+a)^{k_1}} \ze_{n_1-1}(\ora\bfk_{\hskip-2pt 2,j};\ora\bfx_{\hskip-2pt 2,j};a)
 \frac{x_r^{n_r+a}}{(n_r+a)^{k_r}} \ze^{\star}_{n_r}(\ola\bfk_{\hskip-2pt j+1,r-1}; \ola\bfx_{\hskip-2pt j+1,r-1};a)\\
=&\sum_{j=0}^{r-1} (-1)^j \sum_{n_r=1}^{n} \frac{x_r^{n_r+a}}{(n_r+a)^{k_r}} \ze_{n_r-1}(\ora\bfk_{\hskip-2pt j};\ora\bfx_{\hskip-2pt j};a)\ze^{\star}_{n_r}(\ola\bfk_{\hskip-2pt j+1,r-1}; \ola\bfx_{\hskip-2pt j+1,r-1};a)\\
+&\sum_{j=1}^r (-1)^j \sum_{n_r=1}^{n} \frac{x_1^{n_1+a}}{(n_1+a)^{k_1}} \ze_{n_1-1}(\ora\bfk_{\hskip-2pt 2,j};\ora\bfx_{\hskip-2pt 2,j};a)\ze^{\star}_{n_1}(\ola\bfk_{\hskip-2pt j+1}; \ola\bfx_{\hskip-2pt j+1};a).
\end{align*}
Now
\begin{align*}
\ze_{n_r-1}(\ora\bfk_{\hskip-2pt j};\ora\bfx_{\hskip-2pt j};a)=
\left\{
  \begin{array}{ll}
     \ze_{n_r}(\ora\bfk_{\hskip-2pt j};\ora\bfx_{\hskip-2pt j};a)-\frac{x_1^{n_r+a}}{(n_r+a)^{k_1}} \ze_{n_r}(\ora\bfk_{\hskip-2pt 2,j};\ora\bfx_{\hskip-2pt 2,j};a), \quad \ & \hbox{if $j>0$;} \\
   \ze_{n_r}(\ora\bfk_{\hskip-2pt j};\ora\bfx_{\hskip-2pt j};a)(=1) , & \hbox{if $j=0$;}
  \end{array}
\right.
\end{align*}
and similarly,
\begin{align*}
\ze_{n_1-1}(\ora\bfk_{\hskip-2pt 2,j};\ora\bfx_{\hskip-2pt 2,j};a)=
\left\{
  \begin{array}{ll}
     \ze_{n_1}(\ora\bfk_{\hskip-2pt 2,j};\ora\bfx_{\hskip-2pt 2,j};a)-\frac{x_2^{n_1+a}}{(n_1+a)^{k_2}} \ze_{n_1}(\ora\bfk_{\hskip-2pt 3,j};\ora\bfx_{\hskip-2pt 3,j};a), \quad \ & \hbox{if $j>1$;} \\
    \ze_{n_1}(\ora\bfk_{\hskip-2pt 2,j};\ora\bfx_{\hskip-2pt 2,j};a)(=1) , & \hbox{if $j=1$.}
  \end{array}
\right.
\end{align*}
We see that
\begin{align*}
&\sum_{j=0}^r (-1)^j \ze_n(\ora\bfk_{\hskip-2pt j};\ora\bfx_{\hskip-2pt j};a)\ze^{\star}_n(\ola\bfk_{\hskip-2pt j+1}; \ola\bfx_{\hskip-2pt j+1};a)\\
=&\sum_{n_r=1}^{n} \frac{x_r^{n_r+a}}{n_r^{k_r}}  \sum_{j=0}^{r-1} (-1)^j \ze_{n_r}(\ora\bfk_{\hskip-2pt j};\ora\bfx_{\hskip-2pt j};a)\ze^{\star}_{n_r}(\ola\bfk_{\hskip-2pt j+1,r-1}; \ola\bfx_{\hskip-2pt j+1,r-1};a)\\
-&\sum_{n_r=1}^{n} \frac{(x_1x_r)^{n_r+a}}{(n_r+a)^{k_1+k_r}}  \sum_{j=1}^{r-1} (-1)^j \ze_{n_r}(\ora\bfk_{\hskip-2pt  2,j};\ora\bfx_{\hskip-2pt 2,j};a)\ze^{\star}_{n_r}(\ola\bfk_{\hskip-2pt j+1,r-1}; \ola\bfx_{\hskip-2pt j+1,r-1};a)\\
+&\sum_{n_r=1}^{n} \frac{x_1^{n_1+a}}{n_1^{k_1}} \sum_{j=1}^r (-1)^j \ze_{n_1-1}(\ora\bfk_{\hskip-2pt 2,j};\ora\bfx_{\hskip-2pt 2,j};a)\ze^{\star}_{n_1}(\ola\bfk_{\hskip-2pt j+1}; \ola\bfx_{\hskip-2pt j+1};a)\\
-&\sum_{n_r=1}^{n} \frac{(x_1x_2)^{n_1+a}}{(n_1+a)^{k_1+k_2}} \sum_{j=2}^r (-1)^j \ze_{n_1}(\ora\bfk_{\hskip-2pt 3,j};\ora\bfx_{\hskip-2pt 3,j};a)\ze^{\star}_{n_1}(\ola\bfk_{\hskip-2pt j+1}; \ola\bfx_{\hskip-2pt j+1};a) = 0
\end{align*}
by induction. Here we notice that the two negative sums are zero because $r\ge 3$.
\end{proof}

Letting $n\rightarrow \infty$ yields the following corollary involving the usual MPLs.
\begin{cor}
Let $a$ be a complex parameter, $\bfk=(k_1,\ldots,k_r)\in\N^r$, and $\bfx=(x_1,\dotsc,x_r)\in\gD^r$.
If $(k_1,x_1)\neq (1,1)$ and $(k_r,x_r)\neq (1,1)$ then we have
\begin{align}\label{eq-PMPLs3}
\sum_{j=0}^r (-1)^j \Li_{\ora\bfk_{\hskip-2pt j}}(\ora\bfx_{\hskip-2pt j};a)\Li^\star_{\ola\bfk_{\hskip-2pt j+1}}(\ola\bfx_{\hskip-2pt j+1};a)=0,
\end{align}
where $\Li_{\emptyset}(\emptyset;a)=\Li^\star_{\emptyset}(\emptyset;a):=1$.
\end{cor}

Obviously, setting $a=0$ in \eqref{eq-PMPLs2} and \eqref{eq-PMPLs3} gives the Sakugawa-Seki \cite[Thms. 2.10 and 2.13]{SS2016}. Further, using Theorem \ref{thm-pmhns}, we can get the following two theorems.
\begin{thm}\label{thm-nPMPLs1}
Let $a$ be a complex parameter, $\bfk=(k_1,\ldots,k_r)\in\N^r$, and $\bfx=(x_1,\dotsc,x_r)\in\gD^r$. Then
for any $n\in\N$ and $l\in\N_0$, we have
\begin{align}\label{eq-nPMPLs1}
&\sum_{n\geq n_1>\cdots>n_r>0} \prod_{j=1}^r \frac{x_{j}^{n_{j}+l+a}}{(n_{j}+l+a)^{k_{j}}}
=(-1)^r \sum_{j=0}^r (-1)^j \ze_{n+l}(\ora\bfk_{\hskip-2pt j};\ora\bfx_{\hskip-2pt j};a)\ze^\star_l(\ola\bfk_{\hskip-2pt j+1}; \ola\bfx_{\hskip-2pt j+1};a).
\end{align}
\end{thm}

\begin{thm}\label{thm-nPMPLs2}
Let $a$ be a complex parameter, $\bfk=(k_1,\ldots,k_r)\in\N^r$, and $\bfx=(x_1,\dotsc,x_r)\in\gD^r$. Then
for any $n\in\N$ and $l\in\N_0$, we have
\begin{align}\label{eq-nPMPLs2}
&\sum_{n\geq n_1\geq \cdots\geq n_r>0} \prod_{j=1}^r \frac{x_{j}^{n_{j}+l+a}}{(n_{j}+l+a)^{k_{j}}}
=(-1)^r \sum_{j=0}^r (-1)^j \ze^\star_{n+l}(\ora\bfk_{\hskip-2pt j};\ora\bfx_{\hskip-2pt j};a) \ze_l(\ola\bfk_{\hskip-2pt j+1}; \ola\bfx_{\hskip-2pt j+1};a).
\end{align}
\end{thm}

The proofs of Theorems \ref{thm-nPMPLs1} and \ref{thm-nPMPLs2} are completely similar to the proof of Theorem \ref{thm-PMPLs1} and are thus omitted. We leave the detail to the
interested reader. It is clear that setting $a=0$ in \eqref{eq-nPMPLs1} and \eqref{eq-nPMPLs2} yields \eqref{eq-PMPLs1} and \eqref{eq-PMPLs2}, respectively.

\section{Multiple Integrals Associated with Multiple-Labeled Posets}\label{sec:mulInt}
Yamamoto first used a graphical representation to study the MZVs and MZSVs in \cite{Yamamoto2014}. In this last section, we will apply this idea to study the convoluted MZVs and their $M$-variant.

\subsection{Multiple Integrals of MPL Associated with Multiple-Labeled Posets}
\begin{defn}
A \emph{$(r+1)$-poset} is a pair $(X,\delta_X)$, where $X=(X,\leq)$ is
a finite partially ordered set and $\delta_X$ is a map from $X$ to $\{0,\varepsilon_1,\varepsilon_2,\ldots,\varepsilon_r\}$, where
$0<|\varepsilon_{j}|\leq 1$ for all $j=1,\dots,r$.
The $\delta_X$ is called the \emph{label map} of $X$.
We often omit  $\delta_X$ and simply say ``a $(r+1)$-poset $X$''.

Similar to $2$-poset defined in \cite{Yamamoto2014}, a $(r+1)$-poset $(X,\delta_X)$ is called \emph{admissible}
if $\delta_X(x)=\varepsilon \ne 1$ for all minimal
elements $x \in X$ and $\delta_X(x) \ne 0$ for all maximal elements $x \in X$.
\end{defn}

\begin{defn}
For an admissible $(r+1)$-poset $X$, we define the associated integral
\begin{equation}
J(X)=\int_{\Delta_X}\prod_{x\in X}\Omega_{\delta_X(x)}(t_x), \quad
\Delta_X=\bigl\{(t_x)_x\in [0,1]^X \bigm| t_x<t_y \text{ if } x>y\bigr\}
\end{equation}
where
\begin{equation*}
\Omega_0(t)=\frac{dt}{t},\qquad \Omega_{\varepsilon_{j}}(t)=\frac{\eps_{j} dt}{1-\varepsilon_{j} t} \ \text{ for all } j=1,\dots,r.
\end{equation*}
For the empty $(r+1)$-poset, denoted by $\emptyset$, we put $J(\emptyset):=1$.
\end{defn}

Similarly, we also use Hasse diagrams to indicate $(r+1)$-posets, with vertices $\circ$ and ``$\mathop{\bullet}\limits_{\varepsilon}$" corresponding to $\delta(x)=0$ and $\delta(x)=\varepsilon$, respectively. In particular, let
``$\mathop{\circ}\limits_{}^n$" correspond to the 1-form ``$t^{n-1}dt$". The following result is straight-forward.

\begin{pro}\label{prop:shufflrposet}
For non-comparable elements $a$ and $b$ of a $3$-poset $X$, $X^b_a$ denotes the $(r+1)$-poset that is obtained from $X$ by adjoining the relation $a<b$. If $X$ is an admissible $(r+1)$-poset, then the $(r+1)$-poset $X^b_a$ and $X^a_b$ are admissible and
\begin{equation}
J(X)=J(X^b_a)+J(X^a_b).
\end{equation}
\end{pro}

For $\bfk=(k_1,\dotsc,k_s)\in\N^s$ and $\bfet=(\eps_1,\ldots,\eps_s)\in[-1,1]^s$ (admissible or not),
we write
$$
\begin{xy}
{(0,-2) \ar @{{*}.o} (5,3)},
{(0,-2) \ar @/^0.6mm/ @{-}^{(\bfk,\bfeps)} (5,3)}
\end{xy}=
\begin{xy}
{(-2,-14) \ar @{}_{\eps_s} (-2,-14)},
{(0,-12) \ar @{{*}-o} (4,-10)},
{(4,-10) \ar @{.o} (10,-7)},
{(10,-7) \ar @{-} (14,-5)},
{(14,-5) \ar @{.} (20,-2)},
{(20,-2) \ar @{-{*}} (24,0)},
{(24,0) \ar @{}_{\eps_1} (24,0)},
{(24,0) \ar @{-o}(28,2)},
{(28,2) \ar @{.o} (34,5)},
{(0,-11.5) \ar @/^2mm/ @{-}^{k_s} (9,-7)},
{(24,1) \ar @/^2mm/ @{-}^{k_{1}} (33,5)},
\end{xy} \quad
\text{and} \quad
{\scriptstyle \odot}\hskip-4pt\rightsquigarrow\hskip-4pt(\bfk,\bfeps)
=\begin{xy}
{(15,-3) \ar @{}_{\odot} (15,-3)},
{(15,-5) \ar @{}_{\varepsilon_1} (15,-5)},
{(18,-3) \ar @{{}-} (21,0)},
{(21,0) \ar @{{o}.} (24,3)},
{(24,3) \ar @{{o}-} (27,-3)},
{(27,-3) \ar @{{*}-} (30,0)},
{(25,-3) \ar @{}_{\varepsilon_2} (25,-3)},
{(30,0) \ar @{{o}.} (33,3)},
{(33,3) \ar @{{o}-} (35,-2)},
{(37,0) \ar @{.} (41,0)},
{(42,-3) \ar @{{*}-} (45,0)},
{(40,-3) \ar @{}_{\varepsilon_{s}} (40,-3)},
{(45,0) \ar @{{o}.{o}} (48,3)},
{(19,-3) \ar @/_0.6mm/ @{-}_{\hskip-2pt k_1} (24,2)},
{(28,-3) \ar @/_0.6mm/ @{-}_{\hskip-2pt k_{2}} (33,2)},
{(43,-3) \ar @/_0.6mm/ @{-}_{\hskip-2pt k_s} (48,2)},
\end{xy}
$$
where ${\scriptstyle \odot}=\circ$ or $\bullet$.
If $k_i=1$, we understand $\begin{xy}
{(0,-3) \ar @{}_{\eta_i} (0,-3)},
{(0,-2) \ar @{{*}-o} (4,0)},
{(4,0) \ar @{.o} (10,3)},
{(0,-2) \ar @/^2mm/ @{-}^{k_i} (9,3)}
\end{xy}$ as a single ``$\mathop{\bullet}\limits_{\eta_i}$".
We see from \eqref{equ:glInteratedInt} that
\begin{align}
J\Big(\raisebox{-1pt}{\begin{xy}
{(0,-2) \ar @{{*}.o} (4,3)},
{(0,-2) \ar @/^0.6mm/ @{-}^{(\bfk,\bfet)} (4,3)}
\end{xy}}\, \Big)=\Li_{\bfk}\Big(\eta_1,\frac{\eta_2}{\eta_1},\dotsc,\frac{\eta_r}{\eta_{r-1}}\Big)=\Li_{\bfk}\big(\bfw(\bfet)\big).
\end{align}
Here and in what follows we will need the following maps ($\eta_1,\ldots,\eta_r\neq 0$):
\begin{align*}
&\bfu(\eta_1,\eta_2,\ldots,\eta_r):=\Big(\frac{\eta_1}{\eta_2},\ldots,\frac{\eta_{r-1}}{\eta_r},\eta_r\Big),\quad
\bfv(\eta_1,\eta_2,\ldots,\eta_r):=\Big(\frac{\eta_1}{\eta_2},\ldots,\frac{\eta_{r-1}}{\eta_r}\Big),\\
&\bfw(\eta_1,\eta_2,\ldots,\eta_r):=\Big(\eta_1,\frac{\eta_2}{\eta_1},\ldots,\frac{\eta_{r}}{\eta_{r-1}}\Big),\quad
\bfy(\eta_1,\eta_2,\ldots,\eta_r):=\Big(\frac{\eta_2}{\eta_1},\ldots,\frac{\eta_{r}}{\eta_{r-1}}\Big).
\end{align*}

\begin{pro}
Let $n_0:=n\in \N$, $\bfl=(l_1,\ldots,l_s)\in\N^s$ and $\bfeps:=(\varepsilon_1,\dotsc,\varepsilon_s)\in[-1,1]^s$. Then
\begin{multline}\label{F2-MPLs}
J\bigg( \raisebox{4pt}{\begin{xy}
{(13,3) \ar @{}_{n}  (13,3)},
{(15,3) \ar @{} (15,2.5)},
{(15,1.5) \ar @{{o}-} (18,-3)},
{(18,-3) \ar @{{*}-} (18,-3)},
{(28,-1.2) \ar @{}_{\rightsquigarrow\hskip-1pt(\bfl,\bfeps)} (28,-1.2)},
\end{xy}}
\bigg)
=(1-\eps_1^{-n})J\Big({\circ}\hskip-4pt\rightsquigarrow\hskip-4pt(\bfl,\bfeps)\Big)
+\frac{\ze^\star_n(\bfl;\bfu(\bfeps))}{\eps_1^n} \\
 +\sum_{j=2}^s \frac{\ze^\star_n\left(\ora\bfl_{\hskip-3pt j-1};\bfu(\ora\bfeps_{\hskip-3pt j-1})\right)-\ze^\star_n\left(\ora\bfl_{\hskip-3pt j-1};\bfv(\ora\bfeps_{\hskip-3pt j})\right)}{\eps_1^n}\times
J\Big(\circ\hskip-4pt\rightsquigarrow\hskip-4pt(\ora\bfl_{\hskip-3pt j,s},\ora\bfeps_{\hskip-3pt j,s})\Big).
\end{multline}
\end{pro}
\begin{proof} Observe that
\begin{equation*}
\eps \frac{1-t^n}{1-\eps t}=(1-\eps^{-n})\frac{\eps}{1-\eps t}+\frac{1}{\eps^{n}}\sum_{m=1}^n \eps^m t^{m-1}.
\end{equation*}
Hence, by an elementary calculation, we have the following recurrence relation
\begin{align*}
J\bigg( \raisebox{4pt}{\begin{xy}
{(13,3) \ar @{}_{n}  (13,3)},
{(15,3) \ar @{} (15,2.5)},
{(15,1.5) \ar @{{o}-} (18,-3)},
{(18,-3) \ar @{{*}-} (18,-3)},
{(28,-1.2) \ar @{}_{\rightsquigarrow\hskip-1pt(\bfl,\bfeps)} (28,-1.2)},
\end{xy}}
\bigg)
=(1-\eps_1^{-n})J\Big({\circ}\hskip-4pt\rightsquigarrow\hskip-4pt(\bfl,\bfeps)\Big)+
\frac{1}{\eps_1^n} \sum_{n_1=1}^n \frac{\eps_1^{n_1}}{n_1^{k_1}}
J\bigg( \raisebox{4pt}{\begin{xy}
{(13,3) \ar @{}_{n_1}  (13,3)},
{(15,3) \ar @{} (15,2.5)},
{(15,1.5) \ar @{{o}-} (18,-3)},
{(18,-3) \ar @{{*}-} (18,-3)},
{(36,-0.7) \ar @{}_{\rightsquigarrow\hskip-1pt(\ora\bfl_{\hskip-3pt j,s},\ora\bfeps_{\hskip-3pt j,s})} (36,-0.7)},
\end{xy}}
\bigg).
\end{align*}
Then, applying the above recurrence relation, we deduce the desired evaluation.
\end{proof}

\begin{thm}\label{thm-mpls} For any two compositions of positive integers $\bfk=(k_1,\dotsc,k_r)$, $\bfl=(l_1,\dotsc,l_s)$,  $\bfet:=(\eta_1,\dotsc,\eta_r)\in[-1,1]^r$ and $\bfeps:=(\varepsilon_2,\dotsc,\varepsilon_{s})\in[-1,1]^{s-1}$,
\begin{align}\label{Int-Ser-AMPls}
&J\left( \raisebox{-6pt}{\begin{xy}
{(12,-3) \ar @{{}.} (18,3)},
{(12,-3) \ar @{{*}} (12,-3)},
{(18,3) \ar @{{o}-} (21,6)},
{(21,6) \ar @{{o}.} (24,9)},
{(24,9) \ar @{{o}-{*}} (27,3)},
{(12,-3) \ar @/_0.6mm/ @{}_{(\bfk,\bfet)} (18,3)},
{(18,4) \ar @/^0.6mm/ @{-}^{l_1} (23,9)},
{(26.5,6.5) \ar @{}_{\rightsquigarrow\hskip-1pt(\ora\bfl_{\hskip-3pt 2,s},\ora\bfeps_{\hskip-2pt 2,s})} (26.5,6.5)},
\end{xy}}
\right)
=\left\{\Li_{k_1+l_1,\ora\bfk_{\hskip-2pt 2,r}}\left(\bfw(\bfet)\right)-\Li_{k_1+l_1,\ora\bfk_{\hskip-2pt  2,r}}\left(\bfy(\eps_2,\bfet)\right)\right\}
J\Big(\circ\hskip-4pt\rightsquigarrow\hskip-4pt(\ora\bfl_{\hskip-3pt 2,s},\ora\bfeps_{\hskip-2pt 2,s})\Big)\nonumber\\
&\quad+\sum_{j=2}^{s-1} \left\{\ze\left(\big(\bfk;1,\bfy(\bfet)\big)\circledast \Big(\ora\bfl_{\hskip-3pt j};\bfu(\eta_1,\ora\bfeps_{\hskip-3pt j})\Big)^\star\right)-\ze\left(\big(\bfk;1,\bfy(\bfet)\big)\circledast \Big(\ora\bfl_{\hskip-3pt j};\bfv(\eta_1,\ora\bfeps_{\hskip-3pt j+1})\Big)^\star\right) \right\}\nonumber\\
&\quad\quad\quad\quad\times J\Big(\circ\hskip-4pt\rightsquigarrow\hskip-4pt(\ora\bfl_{\hskip-3pt j+1,s},\ora\bfeps_{\hskip-3pt j+1,s})\Big) +\ze\left(\big(\bfk;1,\bfy(\bfet)\big)\circledast \Big(\bfl;\bfu(\eta_1,\ora\bfeps_{\hskip-2pt s})\Big)^\star\right).
\end{align}
\end{thm}
\begin{proof}
The proof of \eqref{thm-mpls} is similar as the proof of Theorem \eqref{thm-mmvs}. Using the iterated integral \eqref{equ:glInteratedInt}, we obtain
\begin{align*}
& J\left( \raisebox{-6pt}{\begin{xy}
{(12,-3) \ar @{{}.} (18,3)},
{(12,-3) \ar @{{*}} (12,-3)},
{(18,3) \ar @{{o}-} (21,6)},
{(21,6) \ar @{{o}.} (24,9)},
{(24,9) \ar @{{o}-{*}} (27,3)},
{(12,-3) \ar @/_0.6mm/ @{}_{(\bfk,\bfet)} (18,3)},
{(18,4) \ar @/^0.6mm/ @{-}^{l_1} (23,9)},
{(26.5,6.5) \ar @{}_{\rightsquigarrow\hskip-1pt(\ora\bfl_{\hskip-3pt 2,s},\ora\bfeps_{\hskip-2pt 2,s})} (26.5,6.5)},
\end{xy}}
\right)
=\su \frac{\ze_{n-1}\Big(\ora\bfk_{\hskip-2pt 2,r};\bfy({\bfet})\Big)}{n^{k_1+l_1}}\eta_1^n
J\bigg( \raisebox{4pt}{\begin{xy}
{(13,3) \ar @{}_{n}  (13,3)},
{(15,3) \ar @{} (15,2.5)},
{(15,1.5) \ar @{{o}-} (18,-3)},
{(18,-3) \ar @{{*}-} (18,-3)},
{(36,-0.7) \ar @{}_{\rightsquigarrow\hskip-1pt(\ora\bfl_{\hskip-3pt 2,s},\ora\bfeps_{\hskip-2pt 2,s})} (36,-0.7)},
\end{xy}}
\bigg).
\end{align*}
Thus, applying \eqref{F2-MPLs} with an elementary calculation, we obtain the formula \eqref{Int-Ser-AMPls}.
\end{proof}

If letting all $\eps_i=\eta_{j}=1$ in \eqref{Int-Ser-AMPls}, then we obtain the ``integral-series" relation of Kaneko--Yamamoto \cite{Yamamoto2014}. If letting all $\eps_i,\eta_{j}\in\{\pm 1\}$, then we also obtain many `integral-series" relations of alternating multiple zeta values. For example, setting $\bfk=(1,1)$ and $\bfl=(1,2)$ with $\eps_2=1$ and $(\eta_1,\eta_2)\in \{\pm 1\}^2$ in \eqref{Int-Ser-AMPls} gives
\begin{align*}
&2{\Li}_{3,1,1}(1,\eta_1,\eta_1\eta_2) +2{\Li}_{3,1,1}(\eta_1,\eta_1,\eta_2)+2{\Li}_{3,1,1}(\eta_1,\eta_1\eta_2,\eta_2)\nonumber\\&\quad+{\Li}_{2,2,1}(\eta_1,\eta_1,\eta_2)+{\Li}_{2,2,1}(\eta_1,\eta_1\eta_2,\eta_2)
+{\Li}_{2,1,2}(\eta_1,\eta_1\eta_2,\eta_2) \\
&=\Li_{2,1,2}(\eta_1,\eta_1\eta_2,1)+\Li_{2,2,1}(\eta_1,1,\eta_1\eta_2)+\Li_{2,3}(\eta_1,\eta_1\eta_2)+\Li_{4,1}(\eta_1,\eta_1\eta_2).
\end{align*}

\subsection{Multiple Integrals of MMVs Associated with $3$-Labeled Posets}

\begin{defn}
For an admissible $3$-poset $X$, we define the associated integral
\begin{equation}
I(X)=\int_{\Delta_X}\prod_{x\in X}\omega_{\delta_X(x)}(t_x), \quad
\Delta_X=\bigl\{(t_x)_x\in [0,1]^X \bigm| t_x<t_y \text{ if } x>y\bigr\}
\end{equation}
where
\begin{equation*}
\omega_{-1}(t)=\frac{2dt}{1-t^2},\quad \omega_0(t)=\frac{dt}{t}, \quad \omega_1(t)=\frac{2tdt}{1-t^2}.
\end{equation*}
For the empty 3-poset, denoted by $\emptyset$, we put $I(\emptyset):=1$.
\end{defn}

It is clear from the definition of MMVs \eqref{defn-mmvs} that
\begin{align}
I\Big(\raisebox{-1pt}{\begin{xy}
{(0,-2) \ar @{{*}.o} (4,3)},
{(0,-2) \ar @/^0.6mm/ @{-}^{(\bfk,\bfet)} (4,3)}
\end{xy}}\, \Big)=M\big(\bfk;\bfp(\bfet)\big).
\end{align}

The following result is obvious and is often used to compute the multiple integrals.
\begin{pro}\label{prop:shuffl3poset}
For non-comparable elements $a$ and $b$ of a $3$-poset $X$, $X^b_a$ denotes the $3$-poset that is obtained from $X$ by adjoining the relation $a<b$. If $X$ is an admissible $3$-poset, then the $3$-poset $X^b_a$ and $X^a_b$ are admissible and
\begin{equation}
I(X)=I(X^b_a)+I(X^a_b).
\end{equation}
\end{pro}

\begin{pro}
Let $n_0=n\in\N$, $\bfl=(l_1,\ldots,l_s)\in\N^s$ and $\bfeps:=(\varepsilon_1,\dotsc,\varepsilon_s)\in\{\pm 1\}^s$. Then
\begin{align}\label{F2-MMVs}
& I\bigg( \raisebox{4pt}{\begin{xy}
{(13,3) \ar @{}_{n}  (13,3)},
{(15,3) \ar @{} (15,2.5)},
{(15,1.5) \ar @{{o}-} (18,-3)},
{(18,-3) \ar @{{*}-} (18,-3)},
{(28,-1.2) \ar @{}_{\rightsquigarrow\hskip-1pt(\bfl,\bfeps)} (28,-1.2)},
\end{xy}}
\bigg)
=\sum_{n\geq n_1\geq \cdots\geq n_s\geq 1}   \prod_{i=1}^{s} \frac{\pi_i}{n_{i}^{l_{i}}}   \\
& + \sum_{j=0}^{s-1}\frac{\eps_{j+1}}{2}
\sum_{\eps=\pm1} \eps I\bigg(\hskip-2pt \raisebox{6pt}{\begin{xy}
{(18,-3) \ar @{{*}-} (18,-3)},
{(16,-5) \ar @{}_{\eps} (16,-5)},
{(18,-3) \ar @{{o}-} (21,0)},
{(21,0) \ar @{{o}.} (24,3)},
{(24,3) \ar @{{o}-{*}} (27,-2)},
{(19,-3) \ar @/_0.6mm/ @{-}_{\hskip-6pt l_{j+1}} (24,2)},
{(26.5,1.2) \ar @{}_{\rightsquigarrow\hskip-1pt(\ora\bfl_{\hskip-3pt j+2,s},\ora\bfeps_{\hskip-2pt j+2,s})} (26.5,1.2)},
\end{xy}}\hskip-2pt
\bigg) \sum_{n\geq n_1\geq \cdots \geq n_{j}\geq 1} \big(1-(-1)^{n_{j}}\big)  \prod_{i=1}^{j}  \frac{\pi_i}{n_i^{l_i}}, \label{equ:NeedModify}
\end{align}
where $\pi_i=1+\eps_i(-1)^{n_{i-1}+n_i}$ and the last sum in \eqref{equ:NeedModify} becomes $1-(-1)^n$ when $j=0$.
\end{pro}
\begin{proof}
The proof is done straightforwardly by computing the multiple integral on the left-hand
side of \eqref{F2-MMVs} as a repeated integral from left to right". Noting that
\begin{equation*}
w_\varepsilon=\frac{dt}{1-t}-\varepsilon \frac{dt}{1+t}\quad (\varepsilon\in\{\pm 1\}),
\end{equation*}
by an elementary calculation, we see that
\begin{equation*}
(1-t^n)w_\varepsilon=\sum_{m=1}^n (1+\varepsilon (-1)^{n+m})t^{m-1}dt-\varepsilon \frac{1-(-1)^n}{2}(w_{-1}-w_1).
\end{equation*}
This quickly leads to the recurrence relation
\begin{multline*}
 I\bigg( \raisebox{4pt}{\begin{xy}
{(13,3) \ar @{}_{n}  (13,3)},
{(15,3) \ar @{} (15,2.5)},
{(15,1.5) \ar @{{o}-} (18,-3)},
{(18,-3) \ar @{{*}-} (18,-3)},
{(28,-1.2) \ar @{}_{\rightsquigarrow\hskip-1pt(\bfl,\bfeps)} (28,-1.2)},
\end{xy}}
\bigg)
=\sum_{n\geq n_1\geq 1} \frac{1+\eps_1(-1)^{n+n_1}}{n_1^{l_1}}
 I\bigg( \raisebox{4pt}{\begin{xy}
{(13,3) \ar @{}_{n_1}  (13,3)},
{(15,3) \ar @{} (15,2.5)},
{(15,1.5) \ar @{{o}-} (18,-3)},
{(18,-3) \ar @{{*}-} (18,-3)},
{(36,-0.7) \ar @{}_{\rightsquigarrow\hskip-1pt(\ora\bfl_{\hskip-3pt 2,s},\ora\bfeps_{\hskip-2pt 2,s})} (36,-0.7)},
\end{xy}}
\bigg)
\nonumber\\
-\eps_1\frac{1-(-1)^n}{2}\bigg\{ I\bigg( \raisebox{4pt}{\begin{xy}
{(18,-3) \ar @{{*}-} (18,-3)},
{(16,-5) \ar @{}_{-1} (16,-5)},
{(18,-3) \ar @{{o}-} (21,0)},
{(21,0) \ar @{{o}.} (24,3)},
{(24,3) \ar @{{o}-{*}} (27,-3)},
{(19,-3) \ar @/_0.6mm/ @{-}_{l_1} (24,2)},
{(26.8,0.2) \ar @{}_{\rightsquigarrow\hskip-1pt(\ora\bfl_{\hskip-3pt 2,s},\ora\bfeps_{\hskip-2pt 2,s})} (26.8,0.2)},
\end{xy}}
\bigg)-I\bigg( \raisebox{4pt}{\begin{xy}
{(18,-3) \ar @{{*}-} (18,-3)},
{(16,-5) \ar @{}_{1} (16,-5)},
{(18,-3) \ar @{{o}-} (21,0)},
{(21,0) \ar @{{o}.} (24,3)},
{(24,3) \ar @{{o}-{*}} (27,-3)},
{(19,-3) \ar @/_0.6mm/ @{-}_{l_1} (24,2)},
{(26.8,0.2) \ar @{}_{\rightsquigarrow\hskip-1pt(\ora\bfl_{\hskip-3pt 2,s},\ora\bfeps_{\hskip-2pt 2,s})} (26.8,0.2)},
\end{xy}}
\bigg)\bigg\}.
\end{multline*}
Repeating the above recurrence relation $r-1$ times, we may easily deduce the desired result.
\end{proof}

\begin{thm}\label{thm-mmvsPoset}
Let $\bfk=(k_1,\dotsc,k_r)\in\N^r$, $\bfl=(l_1,\dotsc,l_s)\in\N^s$, $\bfet:=(\eta_1,\dotsc,\eta_r)\in\{\pm1\}^r$ and $\bfeps:=(\varepsilon_2,\dotsc,\varepsilon_{s})\in\{\pm1\}^{s-1}$. Setting $\underline{\bfet}=\eta_1\cdots\eta_r$, we have
\begin{align}\label{Int-Ser-MMVs}
&I\left( \raisebox{-6pt}{\begin{xy}
{(12,-3) \ar @{{}.} (18,3)},
{(12,-3) \ar @{{*}} (12,-3)},
{(18,3) \ar @{{o}-} (21,6)},
{(21,6) \ar @{{o}.} (24,9)},
{(24,9) \ar @{{o}-{*}} (27,3)},
{(12,-3) \ar @/_0.6mm/ @{}_{(\bfk,\bfet)} (18,3)},
{(18,4) \ar @/^0.6mm/ @{-}^{l_1} (23,9)},
{(26.5,6.5) \ar @{}_{\rightsquigarrow\hskip-1pt(\ora\bfl_{\hskip-3pt 2,s},\ora\bfeps_{\hskip-2pt 2,s})} (26.5,6.5)},
\end{xy}}
\right)
=2M((\bfk;\bfp(\bfet))\circledast (\bfl;0,\underline{\bfet} \bfr(\bfeps))^\star)\nonumber\\
&\quad-\frac{\eps_2(1-\underline{\bfet})}{2}M\left(k_1+l_1,\ora\bfk_{\hskip-2pt 2,r};-1,\bfp(\ora\bfet_{\hskip-2pt 2,r})\right)
\sum_{\eps=\pm1} \eps I\bigg( \raisebox{4pt}{\begin{xy}
{(18,-3) \ar @{{*}-} (18,-3)},
{(16,-5) \ar @{}_{\eps} (16,-5)},
{(18,-3) \ar @{{o}-} (21,0)},
{(21,0) \ar @{{o}.} (24,3)},
{(24,3) \ar @{{o}-{*}} (27,-1)},
{(19,-3) \ar @/_0.6mm/ @{-}_{l_2} (24,2)},
{(26.8,2.2) \ar @{}_{\rightsquigarrow\hskip-1pt(\ora\bfl_{\hskip-3pt 3,s},\ora\bfeps_{\hskip-2pt 3,s})} (26.8,2.2)},
\end{xy}}
\bigg) \nonumber\\
&\quad-\sum_{j=2}^{s-1}\eps_{j+1}(1-\underline{\bfet} \eps_2\cdots\eps_{j})
M\left( \big (\bfk;\bfp(\bfet)\big )\circledast \bfgl_{j}^\star\right)
 \sum_{\eps=\pm1} \eps   I\bigg(\hskip-2pt \raisebox{6pt}{\begin{xy}
{(18,-3) \ar @{{*}-} (18,-3)},
{(16,-5) \ar @{}_{\eps} (16,-5)},
{(18,-3) \ar @{{o}-} (21,0)},
{(21,0) \ar @{{o}.} (24,3)},
{(24,3) \ar @{{o}-{*}} (27,-2)},
{(19,-3) \ar @/_0.6mm/ @{-}_{\hskip-6pt l_{j+1}} (24,2)},
{(26.5,1.2) \ar @{}_{\rightsquigarrow\hskip-1pt(\ora\bfl_{\hskip-3pt j+2,s},\ora\bfeps_{\hskip-2pt j+2,s})} (26.5,1.2)},
\end{xy}}\hskip-2pt
\bigg),
\end{align}
where $\bfgl_{j}=\big(\ora\bfl_{\hskip-3pt j};0,\underline{\bfet}\bfr(\ora\bfeps_{\hskip-3pt j-1}),-1\big)$.
\end{thm}
\begin{proof} According to the iterated integral expressions \eqref{defn-mmvs} of $M$-MPL, we can find that
\begin{align*}
&I\left( \raisebox{-6pt}{\begin{xy}
{(12,-3) \ar @{{}.} (18,3)},
{(12,-3) \ar @{{*}} (12,-3)},
{(18,3) \ar @{{o}-} (21,6)},
{(21,6) \ar @{{o}.} (24,9)},
{(24,9) \ar @{{o}-{*}} (27,3)},
{(12,-3) \ar @/_0.6mm/ @{}_{(\bfk,\bfet)} (18,3)},
{(18,4) \ar @/^0.6mm/ @{-}^{l_1} (23,9)},
{(26.5,6.5) \ar @{}_{\rightsquigarrow\hskip-1pt(\ora\bfl_{\hskip-3pt 2,s},\ora\bfeps_{\hskip-2pt 2,s})} (26.5,6.5)},
\end{xy}}
\right)
=\su \frac{M_{n-1}\Big(\ora\bfk_{\hskip-2pt 2,r};\bfp(\ora\bfet_{\hskip-2pt 2,r})\Big)}{n^{k_1+l_1}}(1+\underline{\bfet} (-1)^n)
I\bigg( \raisebox{4pt}{\begin{xy}
{(13,3) \ar @{}_{n}  (13,3)},
{(15,3) \ar @{} (15,2.5)},
{(15,1.5) \ar @{{o}-} (18,-3)},
{(18,-3) \ar @{{*}-} (18,-3)},
{(36,-0.7) \ar @{}_{\rightsquigarrow\hskip-1pt(\ora\bfl_{\hskip-3pt 2,s},\ora\bfeps_{\hskip-2pt 2,s})} (36,-0.7)},
\end{xy}}
\bigg).
\end{align*}
Then applying \eqref{F2-MMVs} and noting the fact that
\begin{align*}
&(1+\eta (-1)^n)(1+\eps_1(-1)^{n+n_1})(1+\eps_2(-1)^{n_1+n_2})\cdots (1+\eps_r(-1)^{n_{r-1}+n_r})\\
&=(1+\eta (-1)^n)(1+\eta\eps_1(-1)^{n_1})(1+\eta\eps_1\eps_2(-1)^{n_2})\cdots (1+\eta\eps_1\eps_2\cdots\eps_r(-1)^{n_r}),
\end{align*}
with an elementary calculation, we obtain the desired evaluation.
\end{proof}

Let $s=2$ and $\eps_2=\eps$ in Theorem~\ref{thm-mmvsPoset}.  We see that
\begin{align} \label{Int-Ser-MMVs2}
&I\left( \raisebox{-6pt}{\begin{xy}
{(12,-3) \ar @{{*}.} (12,-3)},
{(12,-3) \ar @{{o}.} (18,3)},
{(18,3) \ar @{{o}-} (21,6)},
{(21,6) \ar @{{o}.} (24,9)},
{(24,9) \ar @{{o}-} (27,3)},
{(27,3) \ar @{{*}-} (30,6)},
{(25,3) \ar @{}_{\varepsilon} (25,3)},
{(30,6) \ar @{{o}.} (33,9)},
{(33,9) \ar @{{o}.} (33,9)},
{(12,-3) \ar @/_0.6mm/ @{}_{(\bfk,\bfet)} (18,3)},
{(18,4) \ar @/^1mm/ @{-}^{l_1} (23,9)},
{(28,3) \ar @/_0.6mm/ @{-}_{l_{2}} (33,8)},
\end{xy}}
\right)=2M((\bfk;\bfp(\bfet))\circledast(l_1,l_2;0,\eps\underline{\bfet})^\star)\nonumber\\
&\qquad\qquad\quad\quad\quad-\frac{\eps(1-\underline{\bfet})}{2}M\Big(k_1+l_1,\ora\bfk_{\hskip-2pt 2,r};-1,\bfp(\ora\bfet_{\hskip-2pt 2,r})\Big)
\bigg\{ I\bigg( \raisebox{4pt}{\begin{xy}
{(18,-3) \ar @{{*}-} (18,-3)},
{(16,-5) \ar @{}_{-1} (16,-5)},
{(18,-3) \ar @{{o}-} (21,0)},
{(21,0) \ar @{{o}.} (24,3)},
{(24,3) \ar @{{o}-} (24,3)},
{(19,-3) \ar @/_0.6mm/ @{-}_{l_2} (24,2)},
\end{xy}}
\bigg)-I\bigg( \raisebox{4pt}{\begin{xy}
{(18,-3) \ar @{{*}-} (18,-3)},
{(16,-5) \ar @{}_{1} (16,-5)},
{(18,-3) \ar @{{o}-} (21,0)},
{(21,0) \ar @{{o}.} (24,3)},
{(24,3) \ar @{{o}-} (24,3)},
{(19,-3) \ar @/_0.6mm/ @{-}_{l_2} (24,2)},
\end{xy}}
\bigg)\bigg\}\nonumber\\
&=2M((\bfk;\bfp(\bfet))\circledast(l_1,l_2;0,\eps\underline{\bfet})^\star)
+\eps(1-\underline{\bfet})M\Big(k_1+l_1,\ora\bfk_{\hskip-2pt 2,r};-1,\bfp(\ora\bfet_{\hskip-2pt 2,r})\Big)\ze({\bar l}_2),
\end{align}
where we used the fact that
\begin{align*}
\ze({\bar l}_2)=I\bigg( \raisebox{4pt}{\begin{xy}
{(18,-3) \ar @{{*}-} (18,-3)},
{(16,-5) \ar @{}_{-1} (16,-5)},
{(18,-3) \ar @{{o}-} (21,0)},
{(21,0) \ar @{{o}.} (24,3)},
{(24,3) \ar @{{o}-} (24,3)},
{(19,-3) \ar @/_0.6mm/ @{-}_{l_2} (24,2)},
\end{xy}}
\bigg)-I\bigg( \raisebox{4pt}{\begin{xy}
{(18,-3) \ar @{{*}-} (18,-3)},
{(16,-5) \ar @{}_{1} (16,-5)},
{(18,-3) \ar @{{o}-} (21,0)},
{(21,0) \ar @{{o}.} (24,3)},
{(24,3) \ar @{{o}-} (24,3)},
{(19,-3) \ar @/_0.6mm/ @{-}_{l_2} (24,2)},
\end{xy}}\bigg).
\end{align*}

As an example, setting $\bfk=(2)$ and $(l_1,l_2)=(1,2)$ in \eqref{Int-Ser-MMVs2} further yields
\begin{align*}
&2M(3,2;\eps\eta,\eta)+3M(4,1;\eps\eta,\eta)+3M(4,1;\eta\eps,\eps)+M(3,2;\eta\eps,\eps)\\
&=M(3,2;\eta,\eps\eta)+(1+\eps)M(5;\eta)+\eps(1-\eta)M(3;-1)\ze(\bar 2).
\end{align*}
If setting $\eta_1=\cdots=\eta_r=\eps=-1$ in \eqref{Int-Ser-MMVs2}, we obtain the well-known ``integral-series" identity involving Kaneko--Tsumura's multiple $T$-values \cite[Theorem 4.5]{XuZhao2020a}. For the detailed definition and introduction of multiple $T$-values, please see \cite{KTA2019}.

\bigskip
{\bf Acknowledgments.}  Ce Xu is supported by the National Natural Science Foundation of China (Grant No. 12101008), the Natural Science Foundation of Anhui Province (Grant No. 2108085QA01) and the University Natural Science Research Project of Anhui Province (Grant No. KJ2020A0057). Jianqiang Zhao is supported by the Jacobs Prize from The Bishop's School.

\end{document}